\documentclass[12pt,reqno]{amsproc}
\usepackage{mathrsfs, amsfonts, amsthm, amsmath, amssymb}

 \theoremstyle{plain}

\theoremstyle{plain}
\newtheorem{theorem}{Theorem}[section]

\newtheorem{corollary}[theorem]{Corollary}

\newtheorem{definition}[theorem]{Definition}
\newtheorem{example}[theorem]{Example}

\newtheorem{lemma}[theorem]{Lemma}
\newtheorem{notation}[theorem]{Notation}
\newtheorem{problem}[theorem]{Problem}
\newtheorem{proposition}[theorem]{Proposition}
\newtheorem{remark}[theorem]{Remark}

\makeatletter
\newcommand{\LeftEqNo}{\let\veqno\@@leqno}
\makeatother

 \numberwithin{equation}  {section}
\begin{document}

\

\vspace{-2cm}

\title[A Beurling-Chen-Hadwin-Shen Theorem for Semifinite von Neumann Algebras]{A Beurling-Chen-Hadwin-Shen Theorem for Noncommutative Hardy Spaces Associated with Semifinite von Neumann Algebras with Unitarily Invariant Norms}

\author{Lauren Sager and Wenjing Liu}
\address{Lauren Sager \\
        Department of Mathematics\\
        Saint Anselm College\\
        Manchester, NH 03102;  Email: lbq32@wildcats.unh.edu\\
        Wenjing Liu\\
        Department of Mathematics and Statistics \\
         University of New Hampshire\\
         Durham, NH 03824;   Email: wbs4@wildcats.unh.edu}

\begin{abstract}
 We introduce a class of unitarily invariant, locally $\|\cdot\|_1$-dominating, mutually continuous
 norms with repect to $\tau$ on a von Neumann algebra $\mathcal{M}$ with a faithful, normal, semifinite tracial
 weight $\tau$.  We prove a Beurling-Chen-Hadwin-Shen theorem for $H^\infty$-invariant spaces of $L^\alpha(\mathcal{M},\tau)$,
 where $\alpha$ is a unitarily invariant, locally $\|\cdot\|_1$-dominating, mutually continuous norm with respect to
 $\tau$, and $H^\infty$ is an extension of Arveson's noncommutative Hardy space. We use our main result to characterize the
  $H^\infty$-invariant subspaces of a noncommutative Banach function space $\mathcal I(\tau)$ with the norm $\|\cdot\|_{E }$ on $\mathcal{M}$, the crossed product of a semifinite von Neumann algebra by an action $\beta$, and $B(\mathcal{H})$ for a separable Hilbert space $\mathcal{H}$.
\end{abstract}

\subjclass[2000]{}

\maketitle

\section{Introduction}

    Suppose that $(X,\Sigma, \nu)$ is a localizable measure space  with the finite subset property (i.e. a   measure space is localizable if the multiplication algebra is maximal abelian, and has the finite subset property if for every $A\in\Sigma$ such that $\nu(A)>0$, there exists a $B\in\Sigma$ such that $B\subseteq A$, and $0<\nu(B)<\infty$).  We let $E$ be a two-sided ideal of the set of complex-valued, $\Sigma$-measureable functions on $X$, such that all functions equal almost everywhere with respect to $\nu$ are identified.  If $E$ has a norm $\|\cdot\|_{E}$ such that $(E,\|\cdot\|_E)$ is a Banach lattice, then we call $E$ a Banach function space. (See the work of de Pagter in \cite{depag}).

    We let $\mathcal{M}$ be a von Neumann algebra with a semifinite, faithful,
    normal tracial weight $\tau$.  For every   operator $x\in\mathcal{M}$, we
    define $d_x(\lambda)=\tau(e^{|x|}(\lambda,\infty))$ for every $\lambda\geq 0$ (where $e^{|x|}(\lambda,\infty)$ is
    the spectral projection of $|x|$ on the interval $(\lambda,\infty)$), and $\mu(x)= \inf\{\lambda\geq 0 \,:\, d_x(\lambda)\leq t\}$
     for a given $t\geq 0$.  Consider the set  $\mathcal{I}=\{x\in\mathcal{M}\,:\, \text{$x$ is a finite rank operator in $(\mathcal M,\tau)$ and } \|\mu(x)\|_E<\infty\}$ and let $\|\cdot\|_{\mathcal{I}(\tau)}: \mathcal I\rightarrow [0,\infty)$ be such that
     $\|x\|_{\mathcal{I}(\tau)}=\|\mu(x)\|_E$ for all $x\in\mathcal I$. It is known that $\|\cdot\|_{\mathcal{I}(\tau)}$ defines a norm on $\mathcal I$ (see
     \cite{depag}).
      Denote by $\mathcal{I}(\tau)$ the
      closure of $\mathcal{I}$ under $\|\cdot\|_{\mathcal{I}(\tau)}$.  

    We briefly recall an extension of Arveson's non commutative Hardy space for a semifinite von Neumann algebra. Let $H^\infty$ be a weak* closed unital subalgebra of $\mathcal M$.
    Then
    $\mathcal{D}=H^\infty\cap(H^\infty)^*$ is a von Neumann subalgebra of $\mathcal M$.
     Assume that there also exists a faithful, normal, conditional expectation $\Phi:\mathcal{M}\rightarrow \mathcal{D}$.  Then $H^\infty$ is called a {\em
     semifinite
     non-commutative Hardy space} if (i)  the restriction of $\tau$ on $\mathcal D$ is semifinite; (ii) $\Phi(xy)=\Phi(x)\Phi(y)$ for every
      $x,y\in H^\infty$; (iii) $H^\infty + (H^\infty)^*$ is weak* dense in $\mathcal M$; and (iv)$\tau(\Phi(x))=\tau(x)$ for every positive $x\in\mathcal{M}$.

    We want to ask the following question about the space $\mathcal I(\tau)$:
    \begin{problem}\label{question1.1}
    Consider  a semifinite subdiagonal subalgebra $H^\infty$ of $\mathcal M$  and a closed subspace $\mathcal{K}$ of $\mathcal I(\tau)$ such that $H^\infty \mathcal{K}\subseteq \mathcal{K}$.  How can the subspace $\mathcal{K}$ be characterized?
    \end{problem}

    It can be shown that when $\mathcal{M}$ is diffuse, and $\|\cdot\|_{\mathcal{I}(\tau)}$ is order continuous, the norm $\|\cdot\|_{\mathcal{I}(\tau)} $  on  $\mathcal{I}(\tau)$ is in the family of unitarily invariant, locally $\|\cdot\|_1$-dominating, mutually continuous norms with respect to the tracial weight $\tau$. (See Definition \ref{def3.1}).

    Our goal for this paper is to prove a Beurling-type theorem for a von Neumann algebra with semifinite, faithful,
    normal tracial weight $\tau$, and a unitarily invariant, locally $\|\cdot\|_1$-dominating, mutually continuous norm with respect to $\tau$,
    for example, the Banach function space $\mathcal I(\tau)$ with the norm $\|\cdot\|_{E}$.

    In 1937, J. von Neumann introduced the unitarily invariant norms on $M_n(\mathbb{C})$ as a way to metrize the matrix spaces \cite{vNeumann}.
     He showed that the class of unitarily invariant norms on $M_n(\mathbb{C})$ is in correspondence with the class
      of symmetric guage norms on $\mathbb{C}^n$.  Specifically, he proved that for any unitarily invariant norm $\alpha$, there exists a symmetric guage norm
      $\Psi$ on $\mathbb C^n$ such that for every finite rank operator $A$,  then $\alpha(A)=\Psi(a_1, a_2, \dots,
      a_n)$, where   $\{a_i\}_{1\leq i\leq n}$ is the spectrum of
      $|A|$.

    Since von Neumann's result, these norms have been extended and generalized in different ways.  Schatten defined unitarily invariant norms on 2-sided ideals of the continuous functions on a Hilbert space, $B(\mathcal{H})$ (for example, see \cite{Schatten, Schatten1}).  Chen, Hadwin and Shen defined a class of unitarily invariant, $\|\cdot\|_1$-dominating, normalized norms on a finite von Neumann algebra \cite{CHS}.  Unitarily invariant norms also play an important role in the study of non-commutative Banach function spaces. For more information and history of unitarily invariant norms see Schatten \cite{Schatten}, Hewitt and Ross \cite{HR}, Goldberg and Krein \cite{GK}, or Simon \cite{Simon}.

    A. Beurling proved his classical theorem for invariant subspaces in 1949 \cite{Be}.  We recall the classical Beurling Theorem.  We let $\mathbb{T}$ be the unit circle, and we let $\mu$ be the measure on $\mathbb{T}$ such that $d\mu=\frac{1}{2\pi}d\theta$.  As is standard, we let $L^\infty(\mathbb{T},\mu)$ be the commutative von Neumann algebra on $\mathbb{T}$.  We define $L^2(\mathbb{T},\mu)$ to be the $\|\cdot\|_2$-norm closure of $L^\infty(\mathbb{T},\mu)$, which is a Hilbert space with orthonormal basis $\{z^n:n\in\mathbb{N}\}$.  We define the subspace $H^2=\overline{span(\{z^n:n\geq 0\}}^{\|\cdot\|_2}$ of $L^2(\mathbb{T},\mu)$, and define $H^\infty=H^2\cap L^\infty(\mathbb{T},\mu)$.  It is clear that $L^\infty(\mathbb{T},\mu)$ has a representation onto $B(L^2(\mathbb{T},\mu))$ given by the map $\phi\rightarrow M_\phi$, where $M_\phi$ is given by $M_\phi(f)=\phi f$ for every $f\in L^2(\mathbb{T},\mu)$.  Hence, $L^\infty(\mathbb{T}, \mu)$ and $H^\infty$ act naturally by left (or right) multiplication on $L^2(\mathbb{T},\mu)$.  The classical Beurling Theorem may be stated as follows (for more information, see \cite{B}):
       \emph{Suppose that $\mathcal{W}$ is a nonzero, closed, $H^\infty$ invariant subspace of $H^2$ (namely $z\mathcal{W}\subseteq \mathcal{W}$).  Then $\mathcal{W}=\phi H^2$ for some $\phi\in H^\infty$ such that $|\phi|=1 \,a.e.(\mu)$.}

    The Beurling Theorem has been extended in many ways (see \cite{Bo}, \cite{Halmos}, \cite{He}, \cite{HL}, \cite{Ho} and \cite{Sr}, among others).  One example is as follows: we define $L^p(\mathbb{T},\mu)$ to be the closure of $L^\infty(\mathbb{T},\mu)$ under the $\|\cdot\|_p$-norm.  Also define $H^p=\{f\in L^p(\mathbb{T},\mu) : \int_\mathbb{T} f(e^{i\theta})e^{in\theta}d\mu(\theta)=0 \,\,\forall\,\, n\in\mathcal{N}\}$. The Beurling Theorem may be extended to $H^\infty$-invariant subspaces of the Hardy spaces $H^p$ for $1\leq p\leq \infty$.
Some further extensions of Beurling's theorem can be found in
\cite{BL2} and \cite{CHS}.

    Typical examples of noncommutative Banach functional spaces include so called noncommutative  $L^p$-spaces, $L^p(\mathcal M,\tau)$, associated with   semifinite von Neumann
    algebras.  Suppose $\mathcal{M}$ is a von Neumann algebra with a semifinite, faithful, normal tracial weight $\tau$.  We consider $\mathcal{I}$, the set of elementary operators on $\mathcal{M}$ (when $\mathcal{M}$ is finite, $\mathcal{M}=\mathcal{I}$). We recall the construction of $L^p(\mathcal{M},\tau)$.  When $0<p<\infty$ define a mapping $\|\cdot\|_p:\mathcal{I}\rightarrow [0,\infty)$ by $\|x\|_p=(\tau(|p|))^{1/p}$ where $|x|=\sqrt(x^*x)$ for every $x\in\mathcal{I}$.  It is non-trivial to prove that $\|\cdot\|_p$ is a norm, called the \emph{p-norm}, when $1\leq p<\infty$.  We define the space $L^p(\mathcal{M},\tau)=\overline{\mathcal{I}}^{\|\cdot\|_p}$ for $0<p<\infty$.  When $p=\infty$, we set $L^\infty(\mathcal{M},\tau)=\mathcal{M}$, which acts naturally on $L^p(\mathcal{M},\tau)$ by right or left multiplication.


    In \cite{Sag}, L. Sager extends the work of Blecher and Labuschagne in \cite{BL2} from a finite von Neumann algebra  to von Neumann algebras $\mathcal{M}$ with a \emph{semifinite}, normal, faithful tracial weight $\tau$.  \emph{Suppose $0<p\leq\infty$, and $\mathcal{M}$ is a von Neumann algebra with a semifinite, faithful, normal tracial weight $\tau$.  Let $H^\infty$ be a semifinite subdiagonal subalgebra of $\mathcal{M}$, and $\mathcal{D}=H^\infty \cap (H^\infty)^*$.  Suppose that $\mathcal{K}$ is a closed subspace of $L^p(\mathcal{M},\tau)$ (if $p=\infty$, $\mathcal{K}$ is weak* closed), such that $H^\infty \mathcal{K}\subseteq \mathcal{K}$.  Then there exists a closed subspace $Y\subseteq L^p(\mathcal{M},\tau)$ and a family of partial isometries $\{u_\lambda\}\subseteq \mathcal{M}$ such that $\mathcal{K}=Y\oplus^{row}(\oplus^{row}_{\lambda\in\Lambda} H^p u_\lambda)$, where $Y=[H^\infty_0 Y]_p$, $u_\lambda Y^* =0$ for every $\lambda\in\Lambda$, and the $u_\lambda$ satisfy other conditions.} (See \cite{Sag} for more information.)

    In   \cite{CHS}, Chen, Hadwin and Shen proved a Beurling-type theorem  for unitarily invariant norms on finite von Neumann algebras. A motivation for this paper
    is to extend the result in \cite{CHS} to  the setting of unitarily invariant norms on {\em semifinite} von Neumann algebras. We define the family of unitarily invariant, locally $\|\cdot\|_1$-dominating, mutually continuous norms on the von Neumann
    algebra $\mathcal{M}$ with respect to the semifinite, faithful, normal tracial weight $\tau$. Suppose
     that $\mathcal{M}$ is a von Neumann algebra with a semifinite, faithful normal tracial weight $\tau$.
       We let $\mathcal{I}$ be the set of finite rank operators in $(\mathcal{M},\tau)$.    A norm $\alpha: \mathcal{I}\rightarrow [0,\infty)$ is a unitarily invariant, locally $\|\cdot\|_1$-dominating, mutually continuous norm with respect to $\tau$ if $\alpha$ is a norm for which the following conditions hold:
        \begin{enumerate}
            \item[(i)] for any unitaries $u,v\in\mathcal{M}$ and $x\in\mathcal{I}$, $\alpha(uxv)=\alpha(x)$;
            \item[(ii)] for every projection $e\in\mathcal{M}$ with $\tau(e)<\infty$ and any $x\in\mathcal{I}$, there exists $0<c(e)<\infty$ such that $\alpha(exe)\leq c(e) \|exe\|_1$;
            \item[(iii)]
              \begin{enumerate}
                \item[(a)] if $\{e_\lambda\}$ is an increasing net of projections in $\mathcal{I}$ such that $\tau(e_\lambda x-x)\rightarrow 0$ for every $x\in\mathcal{I}$, then $\alpha(e_\lambda x-x)\rightarrow 0$ for every $x\in\mathcal{I}$;
               \item[(b)] if $\{e_\lambda\}$ is a net of projections in $\mathcal{I}$ such that $\alpha(e_\lambda)\rightarrow 0$, then $\tau(e_\lambda)\rightarrow 0$.
             \end{enumerate}
        \end{enumerate}
      Chen, Hadwin and Shen's family of norms in \cite {CHS} is a subset of this family of norms.
      We also show that the norm $\|\cdot\|_{I(\tau)}$ on a Banach function space $\mathcal I(\tau)$ is a unitarily invariant, $\|\cdot\|_1$-dominating, mutually continuous norm.

    However, many of the methods used by Chen, Hadwin and Shen no longer apply when $\mathcal{M}$ is a semifinite von Neumann algebra.  We use a similar method to extend their theorem as in Sager's work on $L^p(\mathcal{M},\tau)$ spaces (see \cite{Sag}).  We therefore prove a series of density lemmas for the $L^\alpha(\mathcal{M},\tau)$ spaces.

    {\renewcommand{\thetheorem}{\ref{lemma3.2}}
        \begin{lemma}\label{lemma3.2}
        Suppose $\mathcal{M}$ is a von Neumann algebra with a faithful, normal, semifinite tracial weight $\tau$, and that $H^\infty$ is a semifinite, subdiagonal subalgebra of $\mathcal{M}$.  Suppose also that $\alpha$ is a unitarily invariant, locally $\|\cdot\|_1$-dominating, mutually continuous norm with respect to $\tau$.  Assume that $\mathcal{K}$ is a closed subspace of $L^\alpha (\mathcal{M},\tau)$ such that $H^\infty \mathcal{K}\subseteq\mathcal{K}$. Then the following hold:
        \begin{enumerate}
            \item $\mathcal{K}\cap \mathcal{M} = \overline{\mathcal{K}\cap\mathcal{M}}^{w^*} \cap L^\alpha(\mathcal{M},\tau)$
            \item $\mathcal{K}=[\mathcal{K}\cap \mathcal{M}]_\alpha$
        \end{enumerate}
    \end{lemma}
    }

    {\renewcommand{\thetheorem}{\ref{lemma3.3}}
        \begin{lemma}\label{lemma3.3}
        Suppose $\mathcal{M}$ is a  von Neumann algebra with a faithful, normal, semifinite tracial weight $\tau$, and suppose that $\alpha$ is a unitarily invariant, locally $\|\cdot\|_1$-dominating, mutually continuous norm with respect to $\tau$.  Let $H^\infty$ be a semifinite, subdiagonal subalgebra of $\mathcal{M}$.  Assume that $\mathcal{K}$ is a weak* closed subspace of $\mathcal{M}$ such that $H^\infty \mathcal{K}\subseteq \mathcal{K}$. Then
            $$\mathcal{K}=\overline{[\mathcal{K}\cap L^\alpha(\mathcal{M},\tau)]_\alpha \cap \mathcal{M}}^{w^*}.$$
    \end{lemma}
    }

    {\renewcommand{\thetheorem}{\ref{lemma3.4}}
        \begin{lemma}\label{lemma3.4}
        Suppose $\mathcal{M}$ is a semifinite von Neumann algebra with a faithful, normal tracial weight $\tau$, and suppose that $\alpha$ is a unitarily invariant, locally $\|\cdot\|_1$-dominating, mutually-continuous norm with respect to $\tau$.  Let $H^\infty$ be a semifinite, subdiagonal subalgebra of $\mathcal{M}$.  Assume that $S$ is a subset of $\mathcal{M}$ such that $H^\infty S\subseteq S$.  Then
            $$[S\cap L^\alpha(\mathcal{M},\tau)]_\alpha= [\overline{S}^{w^*}\cap L^\alpha (\mathcal{M}, \tau)]_\alpha.$$
    \end{lemma}
    }

    Follow these results, we are able to prove a noncommutative Beurling-Chen-Hadwin-Chen theorem for unitarily invariant, $\|\cdot\|_1$-dominating, mututally continuous with respect to $\tau$ norms on a von Neumann algebra $\mathcal{M}$ with a semifinite, faithful, normal tracial weight $\tau$.

    {\renewcommand{\thetheorem}{\ref{theorem3.1}}
        \begin{theorem}\label{theorem3.1}
        Let $\mathcal{M}$ be a von Neumann algebra with a faithful, normal semifinite tracial weight $\tau$, and $H^\infty$ be a semifinite subdiagonal subalgebra of $\mathcal{M}$.  Let $\alpha$ be a  unitarily invariant, locally $\|\cdot\|_1$-dominating, mutually continuous norm with respect to $\tau$.  Let $\mathcal{D}=H^\infty\cap (H^\infty)^*$.  Assume that $\mathcal{K}$ is a closed subspace of $L^\alpha(\mathcal{M},\tau)$ such that $H^\infty \mathcal{K}\subseteq \mathcal{K}$.  Then, there exist a closed subspace $Y$ of $L^\alpha(\mathcal{M},\tau)$ and a family $\{u_\lambda\}$ of partial isometries in $\mathcal{M}$ such that
            \begin{enumerate}
                \item[(i)] $u_\lambda Y^*=0$ for every $\lambda\in\Lambda$;\
                \item[(ii)] $u_\lambda u_\lambda^*\in\mathcal{D}$, and $u_\lambda u_\mu^*=0$ for every $\lambda, \mu\in \Lambda$ with $\lambda\neq\mu$;
                \item[(iii)] $Y=[H^\infty_0 Y]_\alpha$;
                \item[(iv)] $\mathcal{K}=Y\oplus^{row}(\oplus^{row}_{\lambda\in\Lambda} H^\alpha u_\lambda)$.
            \end{enumerate}
    \end{theorem}
    }

    We can fully characterize $\mathcal{K}$ in the case when $K\subseteq L^\alpha(\mathcal{M},\tau)$ is $\mathcal{M}$-invariant.

    {\renewcommand{\thetheorem}{\ref{corollary5.5}}
        \begin{corollary} \label{corollary5.5}
            Suppose that $\mathcal{M}$ is a von Neumann algebra with a faithful, normal, semifinite tracial weight $\tau$.  Let $\alpha$ be a unitarily invariant, locally $\|\cdot\|_1$-dominating, mutually continuous norm with respect to $\tau$.  Let $\mathcal{K}$ be a subset of $L^\alpha$ such that $\mathcal{M}\mathcal{K}\subseteq\mathcal{K}$.  Then there exists a projection $q$ with $\mathcal{K}=\mathcal{M}q$.
        \end{corollary}
    }

    Furthermore, when $\mathcal{M}$ is a factor, we can weaken the conditions on $\alpha$.

    {\renewcommand{\thetheorem}{\ref{corollary6.2}}
        \begin{corollary} \label{corollary6.2}
        Suppose $\mathcal{M}$ is a factor with a faithful, normal tracial weight $\tau$.  Let $\alpha:\mathcal{I}\rightarrow[0,\infty)$, where $\mathcal{I}$ is the set of elementary operators in $\mathcal{M}$, be a unitarily invariant norm such that any net $\{e_\lambda\}$ in $\mathcal{M}$ with $e_\lambda \uparrow I$ in the weak* topology implies that $\alpha((e_\lambda - I)x)\rightarrow 0$.  Let $\mathcal{D}=H^\infty\cap (H^\infty)^*$.  Assume that $\mathcal{K}$ is a closed subspace of $L^\alpha(\mathcal{M},\tau)$ such that $H^\infty \mathcal{K}\subseteq \mathcal{K}$.  Then, there exist a closed subspace $Y$ of $L^\alpha(\mathcal{M},\tau)$ and a family $\{u_\lambda\}$ of partial isometries in $\mathcal{M}$ such that
            \begin{enumerate}
                \item[(i)] $u_\lambda Y^*=0$ for every $\lambda\in\Lambda$;\
                \item[(ii)] $u_\lambda u_\lambda^*\in\mathcal{D}$, and $u_\lambda u_\mu^*$ for every $\lambda, \mu\in \Lambda$ with $\lambda\neq\mu$;
                \item[(iii)] $Y=[H^\infty_0 Y]_\alpha$;
                \item[(iv)] $\mathcal{K}=Y\oplus^{row}(\oplus^{row}_{\lambda\in\Lambda} H^\alpha u_\lambda)$.
            \end{enumerate}
    \end{corollary}
    }

    Similar to Sager's result in \cite{Sag} for $L^p$ spaces, we prove a Beurling-Chen-Hadwin-Shen theorem for the crossed product of a von Neumann algebra $\mathcal{M}$ by a trace-preserving action $\beta$ with a unitarily invariant, locally $\|\cdot\|_1$-dominating, mutually continuous with respect to the trace $\tau$.

    Sager proved in \cite{Sag} that, \emph{given a von Neumann algebra $\mathcal{M}$ with a semifinite, faithful, normal tracial state $\tau$,
    and a trace-preserving *-automorphism $\beta$ of $\mathcal{M}$, consider the crossed product of $\mathcal{M}$ by the action $\beta$, $\mathcal{M}\rtimes_\beta \mathbb{Z}$, and the extended semifinite, faithful, normal tracial state $\tau$.  Let $H^\infty$ be the weak *-closed non-self-adjoint subalgebra $\mathcal{M}\rtimes_\beta \mathbb{Z}_+$ of $\mathcal{M}\rtimes_\beta\mathbb{Z}$.  Then $H^\infty$ is a seimifinite subdiagonal subalgebra.  Let $0<p<\infty$, and $\mathcal{K}$ be a closed subspace of $L^p(\mathcal{M}\rtimes_\beta \mathbb{Z}, \tau)$ such that $H^\infty \mathcal{K}\subseteq \mathcal{K}$.  Then there exist a projection $q\in\mathcal{M}$ and a family $\{u_\lambda\}$ of partial isometries in in $\mathcal{M}\rtimes_\beta\mathbb{Z}$ which satisfy:
        \begin{enumerate}
            \item[(i)] $u_\lambda =0$ for every $\lambda\in\Lambda$;
            \item[(ii)] $u_\lambda u_\lambda^*\in\mathcal{M}$ and $u_\lambda u_\mu^*=0$ for every $\lambda,\mu\in\Lambda$ where $\lambda\neq\mu$;
            \item[(iii)] $\mathcal{K}=(L^p(\mathcal{M}\rtimes_\beta \mathbb{Z}, \tau)q)\oplus^{row} (\oplus^{row}_{\lambda\in\Lambda} H^p u_\lambda)$.
        \end{enumerate}}

    We are able to prove a similar result, but for any $\alpha$, a unitarily invariant, locally $\|\cdot\|_1$-dominating, mutually continuous norm with respect to $\tau$.

    {\renewcommand{\thetheorem}{\ref{corollary6.5}}
        \begin{corollary}\label{corollary6.5}
            Suppose that $\mathcal{M}$ is a von Neumann algebra with a semifinite, faithful, normal tracial weight $\tau$.  Let $\alpha$ be a unitarily invariant, locally $\|\cdot\|_1$-dominating, mutually continuous norm with respect to $\tau$, and $\beta$ be a trace-preserving, *-automorphism of $\mathcal{M}$.  Consider the crossed product of $\mathcal{M}$ by an action $\beta$, $\mathcal{M}\rtimes_\beta \mathbb{Z}$. Still denote the semifinite, faithful, normal, extended tracial weight on $\mathcal{M}\rtimes_\beta \mathbb{Z}$ by $\tau$.

            Denote by $H^\infty$ the weak *-closed nonself-adjoint subalgebra in $\mathcal{M}\rtimes_\beta \mathbb{Z}$ which is generated by $\{\Lambda(n)\Psi(x) : x\in\mathcal{M}, n\geq 0\}$.  Then $H^\infty$ is a semifinite subdiagonal sublagebra of $\mathcal\rtimes_\beta \mathbb{Z}$.

            Let $\mathcal{K}$ be a closed subspace of $L^\alpha(\mathcal{M}\rtimes_\beta\mathbb{Z},\tau)$ such that $H^\infty \mathcal{K}\subseteq \mathcal{K}$.  Then there exist a projection $q$ in $\mathcal{M}$ and a family $\{u_\lambda\}_{\lambda\in\Lambda}$ of partial isometries in $\mathcal{M}\rtimes_\beta \mathbb{Z}$ which satisfy

            \begin{enumerate}
                \item[(i)] $u_\lambda q=0$ for all $\lambda\in\Lambda$;
                \item[(ii)] $u_\lambda u_{\lambda}^*\in\mathcal{M}$ and $u_\lambda u_{\mu}^*=0$ for all $\lambda,\mu\in\Lambda$ with $\lambda\neq \mu$;
                \item[(iii)] $\mathcal{K}=(L_\alpha(\mathcal{M}\rtimes_\beta \mathbb{Z})q)\otimes^{row}(\otimes^{row}_{\lambda\in\Lambda} H^\alpha u_\lambda)$.
            \end{enumerate}
        \end{corollary}
    }

    As $B(\mathcal{H})$ is a factor and can be realized as the crossed product, we can also weaken the conditions on $\alpha$ when $\mathcal{M}=B(\mathcal{H})$. Additionally, we can fully characterize the $H^\infty$ invariant subspace.

    {\renewcommand{\thetheorem}{\ref{corollary6.6}}
        \begin{corollary}\label{corollary6.6}
            Suppose $\mathcal{H}$ is a separable Hilbert space with an orthonormal base $\{e_m\}_{m\in\mathbb{Z}}$, and let
                $$H^\infty=\{x\in B(\mathcal{H}):\langle xe_m, e_n\rangle=0, \forall n<m\}$$
            be the lower triangular subalgebra of $B(\mathcal{H})$. Then $\mathcal{D}=H^\infty \cap (H^\infty)^*$ is the diagonal subalgebra of $B(\mathcal{H})$.

            Suppose $\alpha:\mathcal{I}\rightarrow[0,\infty)$, where $\mathcal{I}$ is the set of elementary operators in $\mathcal{M}$, is an unitarily invariant norm such that any net $\{e_\lambda\}$ in $\mathcal{M}$ with $e_\lambda \uparrow I$ in the weak* topology implies that $\alpha((e_\lambda - I)x)\rightarrow 0$.

            Assume that $\mathcal{K}$ is a closed subspace of $H^\alpha$ such that $H^\infty\mathcal{K}\subseteq \mathcal{K}$.  Then there exists $\{u_\lambda\}_{\lambda\in\Lambda}$, a family of partial isometries in $H^\infty$ which satisfy
            \begin{enumerate}
                \item[(i)] $u_\lambda u_\lambda^*\in\mathcal{D}$ and $u_\lambda u_\mu^*=0$ for every $\lambda, \mu \in \Lambda$ such that $\lambda\neq\mu$;
                \item[(ii)] $\mathcal{K}=\oplus^{row}_{\lambda\in\Lambda} H^\alpha u_\lambda$.
            \end{enumerate}
        \end{corollary}
    }

    Additionally, we prove a result for a Banach function space $E$ with norm $\|\cdot\|_{E(\tau)}$ and provide an answer for Problem \ref{question1.1}.

    {\renewcommand{\thetheorem}{\ref{corollary6.1}}
        \begin{corollary}\label{corollary6.1}
        Suppose that $\mathcal I(\tau)$ is a Banach function space on the diffuse von Neumann algebra $\mathcal{M}$ with order continuous norm $\|\cdot\|_{\mathcal{I}(\tau)}$.
        Let $\mathcal{D}=H^\infty\cap (H^\infty)^*$.  Assume that $\mathcal{K}$ is a closed subspace of $\mathcal I(\tau)$ such
        that $H^\infty \mathcal{K}\subseteq \mathcal{K}$.  Then, there exist a closed subspace $Y$ of $\mathcal I(\tau)$ and a family $\{u_\lambda\}$ of partial isometries in $\mathcal{M}$ such that
            \begin{enumerate}
                \item[(i)] $u_\lambda Y^*=0$ for every $\lambda\in\Lambda$;\
                \item[(ii)] $u_\lambda u_\lambda^*\in\mathcal{D}$, and $u_\lambda u_\mu^*$ for every $\lambda, \mu\in \Lambda$ with $\lambda\neq\mu$;
                \item[(iii)] $Y=[H^\infty_0 Y]_\alpha$;
                \item[(iv)] $\mathcal{K}=Y\oplus^{row}(\oplus^{row}_{\lambda\in\Lambda} H^{\mathcal{I}(\tau)} u_\lambda)$.
            \end{enumerate}
    \end{corollary}
    }

    We begin in section 2 by discussing the background definitions and preliminary results.  In section 3, we define the class of unitarily invariant, $\|\cdot\|_1$-dominating, mutually continuous norms, which we call the class of $\alpha$-norms.  We discuss the non-commutative Banach function space setting and other applications of $\alpha$-norms.  In section 4, we dicuss Arveson's non-commutative Hardy space.  We prove our main result, a Beurling-Chen-Hadwin-Shen Theorem for $\alpha$-norms, in section 5.  We finally apply our main result to our examples and crossed products in section 6.

\section{Preliminaries and Notation}

In the following section, we give some useful and necessary defintions and results  for a von Neumann algebra with a faithful, normal, semifinite tracial weight.  We also discuss the space of operators affiliated with a von Neumann algebra with a faithful, normal, semifinite tracial weight.

\subsection{Weak* Topology}

    Let $\mathcal{M}$ be a von Neumann algebra with a predual $\mathcal{M}_\#$.  We recall that the weak* topology on $\mathcal{M}$, $\sigma(\mathcal{M},\mathcal{M}_\#)$, is the topology on $\mathcal{M}$ induced by the predual space $\mathcal{M}_\#$.  The following result on weak* topology convergence is useful (see, for instance, Theorem 1.7.8 in \cite{Sakai}).

    \begin{lemma} \label{lemma2.0}
        Let $\mathcal{M}$ be a von Neumann algebra.  If $\{e_\lambda\}_{\lambda\in\Lambda}$ is a net of projections in $\mathcal{M}$ converging to $I$ in the weak* topology, then $e_\lambda x$, $xe_\lambda$, and $e_\lambda x e_\lambda$ converge to $x$ in the weak* topology for all $x$ in $\mathcal{M}$.
    \end{lemma}

\subsection{Semifinite von Neumann Algebras}
    Let $\mathcal{M}$ be a von Neumann algebra.  We let $\mathcal{M}^+$ be the positive part of $\mathcal{M}$.  Recall the defintion of a tracial weight $\tau$ on $\mathcal{M}$:
    A mapping $\tau:\mathcal{M}^+\rightarrow[0,\infty]$ is a \emph{tracial weight} on $\mathcal{M}$ if
        \begin{enumerate}
            \item $\tau(x+y)=\tau(x)+\tau(y)$ for $x,y \in \mathcal{M}^+$;
            \item $\tau(ax)=a\tau(x)$ for every $x\in \mathcal{M}^+$ and $a\in[0,\infty]$; and
            \item $\tau(xx^*)=\tau(x^* x)$ for every $x\in\mathcal{M}$.
        \end{enumerate}
    Such a $\tau$ is called \emph{normal} if it is weak* topology continuous; \emph{faithful} if, given $a\in\mathcal{M}^+$, $\tau(a^* a)=0$ implies that $a=0$; \emph{finite} if $\tau(I)<\infty$; and \emph{semifinite} if for any nonzero $x\in\mathcal{M}^+$, there exists a nonzero $y\in\mathcal{M}^+$ such that $\tau(y)<\infty$, and $y\leq x$.  A von Neumann algebra $\mathcal{M}$ for which a faithful, normal, semifinite tracial weight $\tau$ exists is called \emph{semifinite}.

\subsection{Operators affiliated with $\mathcal{M}$}
    Given a von Neumann algebra $\mathcal{M}$ with a semifinite, faithful, normal tracial weight $\tau$ acting on a Hilbert space $\mathcal{H}$, a \emph{measure topology} on $\mathcal{M}$ is given by the system of neighborhoods $U_{\delta,\epsilon}=\{a\in\mathcal{M}: \|ap\|\leq\epsilon \text{ and } \tau(p^{\perp})\leq\delta \text{ for some projection }p\in\mathcal{M}\}$ for any $\epsilon,\delta>0$ (for more details see \cite{Nelson}).  We say that $a_n$ is \emph{Cauchy in measure} if, given $\epsilon$ and $\delta>0$, there exists an $n_0$ such that if $n,m\geq n_0$, then $a_n-a_m$ is in $U_{\delta,\epsilon}$.

    \begin{definition}
    Let $\widetilde{\mathcal{M}}$ denote the algebra of closed, densely defined (possibly unbounded) operators on $\mathcal{H}$ affiliated with $\mathcal{M}$.
    \end{definition}

    \begin{remark}$\widetilde{\mathcal{M}}$ is also the closure of $\mathcal{M}$ in the measure topology (see \cite{Nelson} for more information).\
    \end{remark}

\section{Unitarily invariant norms and examples}

    In this section, we introduce a class of unitarily invariant, locally $\|\cdot\|_1$-dominating, mutually continuous norms on semifinite von Neumann algebras.  We also introduce interesting examples from this class.

\subsection{$L^\alpha$ spaces of semifinite von Neumann algebras}
    Suppose that $\mathcal{M}$ is a von Neumann algebra with a semifinite, faithful, normal tracial state $\tau$.
     We then let 
        $$\mathcal{I}=span\{xey: x,y\in\mathcal{M}, e\in\mathcal{M}, e=e^2=e^* \text{ with }\tau(e)<\infty\}$$ 
    be the set of elementary operators of $(\mathcal{M},\tau)$ (see Remark 2.3 in \cite{Se}).  For each $1\le p<\infty$ , we define  the $\|\cdot\|_p$-norm on $\mathcal{I}$ by $$\|x\|_p=(\tau(|x|^p))^{1/p}  \qquad \text{ for every $x\in\mathcal{I}$}.$$
It is a non-trivial fact that the mapping $\|\cdot\|_p$ defines a norm on $\mathcal I$.     We let $L^p(\mathcal M,\tau)$ denote the completion of $\mathcal I$ with respect to the $\|\cdot\|_p$-norm.
    \begin{definition} \label{def3.1}
    We call a norm $\alpha:\mathcal{I}\rightarrow [0,\infty)$ a \emph{unitarily invariant, locally $\|\cdot\|_1$-dominating, mutually continuous norm with respect to $\tau$} on $\mathcal{I}$ if it satisfies the following characteristics:
    \begin{enumerate}
        \item 
        $\alpha$ is \emph{unitarily invariant} if for all unitaries $u,v$  in $\mathcal{M}$ and every $x$ in $\mathcal{I}$, $\alpha(uxv)=\alpha(x)$;
        \item $\alpha$ is \emph{locally $\|\cdot\|_1$-dominating} if for every projection $e$ in $\mathcal{M}$ with $\tau(e)<\infty$, there exists $0<c(e)<\infty$ such that $\alpha(exe)\geq c(e)\|exe\|_1$ for  every $x\in\mathcal{I}$;
        \item $\alpha$ is \emph{mutually continuous with respect to} $\tau$; namely
            \begin{enumerate}
            \item[(a)] If $\{e_\lambda\}$ is an increasing net of projections in $\mathcal{I}$ such that  $\tau(e_\lambda x-x)\rightarrow 0$ for every $x\in \mathcal{I}$, then   $\alpha(e_\lambda x-x)\rightarrow 0$ for every $x\in \mathcal{I}$.  Or, equivalently, if $\{e_\lambda\}$ is a net of projections in $\mathcal{I}$ such that $e_\lambda\rightarrow I$ in the weak* topology, then $\alpha(e_\lambda x-x)\rightarrow 0$ for every $x\in \mathcal{I}$.
            \item[(b)] If $\{e_\lambda\}$ is a net of projections in $\mathcal{I}$ such that $\alpha(e_\lambda)\rightarrow 0$, then $\tau(e_\lambda)\rightarrow 0$.
            \end{enumerate}
    \end{enumerate}

    \end{definition}

    \begin{definition}\label{def3.2}
        Let $\mathcal{M}$ be a von Neumann algebra with a semifinite, faithful, normal tracial weight $\tau$.  Suppose $\mathcal{I}=span\{\mathcal{M}e\mathcal{M}: e=e^2=e^* \in \mathcal M \text{ such that } \tau(e)<\infty\}$ is the set of all elementary operators in $\mathcal{M}$.  Suppose $\alpha$ is a unitarily invariant, locally $\|\cdot\|_1$-dominating, mutually continuous norm with respect to $\tau$ on $\mathcal{I}$.  We define $L^\alpha(\mathcal{M},\tau)$ to be the completion of $\mathcal{I}$ under $\alpha$, namely,
        $$L^\alpha(\mathcal{M},\tau)=\overline{\mathcal{I}}^\alpha.$$
    \end{definition}

    \begin{notation}
    We will denote by $[S]_\alpha$ the completion, with respect to the norm $\alpha$,  of a set $S$ in $\mathcal{M}$.
    \end{notation}


    \begin{lemma}\label{lemma3.4.3}
        Suppose $\mathcal{M}$ is a von Neumann algebra with a semifinite, faithful, normal tracial weight $\tau$, and let $\alpha$ be a unitarily invariant, locally $\|\cdot\|_1$-dominating, mutually continuous norm with respect to $\tau$.  Then for any $x\in L^\alpha(\mathcal{M},\tau)$, and $a,b\in\mathcal{M}$,
            $$\alpha(axb)\leq \|a\|\alpha(x)\|b\|.$$
    \end{lemma}

    \begin{proof}
        The proof is included here for completeness. It suffices to show that   for any $x\in  \mathcal{I} $, and $a,b\in\mathcal{M}$,
            $$\alpha(axb)\leq \|a\|\alpha(x)\|b\|.$$

    Without loss of generality, we might assume that $\|a\|<1$. By the Russo-Dye Theorem, there exist a positive integer $n$ and unitary elements $u_1,\ldots, u_n$ in $\mathcal M$ such that $a=(u_1+\cdots +u_n)/n$.
        Therefore,
        $$\alpha(ax)=\alpha((u_1+\cdots +u_n)x)/n\le \alpha(x)$$
        since $\alpha$ is unitarily invariant.
        So, $\alpha(ax)\leq \|a\|\alpha(x)$ for every $a\in\mathcal{M}$.

        It may be proved similarly that $\alpha(xb)\leq \alpha(x)\|b\|$ for every $b\in\mathcal{M}$.
    \end{proof}


%

\subsection{Examples of unitarily-invariant, locally $\|\cdot\|_1$-
dominating, mutually continuous norms}


    \begin{remark}
        It is trivial to show that the $\|\cdot\|_p$-norms of $ \mathcal{M} $ with $1\leq p< \infty$ for a semifinite von Neumann algebra $\mathcal{M}$ with a faithful, normal, semifinite tracial weight $\tau$ are unitarily equivalent, $\|\cdot\|_1$-dominating, mutually continuous norms with respect to $\tau$ on $\mathcal{M}$.
    \end{remark}

    \begin{remark}
        It is also trivial to show that a continuous, unitarily invariant, normalized, $\|\cdot\|_1$-dominating norm on a finite von Neumann algebra $\mathcal{M}$ as given in   \cite{CHS}  is a unitarily invariant, locally $\|\cdot\|_1$-dominating, mutually continuous norm with respect to $\tau$ on $\mathcal{M}$.
    \end{remark}

    \begin{proposition} \label{prop3.5}
    Suppose that $\mathcal{M}$ is a semifinite factor, and $\alpha:\mathcal{I}\rightarrow[0,\infty)$ is a unitarily invariant norm satisfying that, if $\{e_\lambda\}$ is a net in $\mathcal{M}$ with $e_\lambda \rightarrow I$ in the weak* topology, then $\alpha(e_\lambda x-x)\rightarrow 0$ for each $x\in\mathcal I$.  Then $\alpha$ is a unitarily invariant, locally $\|\cdot\|_1$-dominating, mutually continuous norm with respect to $\tau$.
    \end{proposition}

    \begin{proof}
        By assumption, $\alpha$ is unitarily invariant.

    Let $e$ be projection in $\mathcal M$ such that $\tau(e)<\infty$.   Let $x=exe$ be an element in $e\mathcal M e$, which we denote by $\mathcal{M}_e$.  As $|x|\le \|x\|e$, we have that $\alpha(x)=\alpha(|x|)\leq \|x\|\alpha(e)$.  Note $\mathcal{M}_e$ is a finite factor with a tracial state $\tau_e$, defined by $\tau_e(y)=\tau(y)/\tau(e)$ for all $y\in \mathcal M_e$. By the Dixmeier Approximation Property, for every  $\epsilon>0$, there exist $c_1, c_2,\dots, c_n$ in $[0,1]$ with  $\sum_{i=1}^n c_i=1$  and unitaries $u_1, u_2, \dots, u_n$ in $e\mathcal{M}e$ such that $\|\tau_e(|x|)e-\sum_{i=1}^n c_i u_i x u_i^*\|<\epsilon$.  Therefore,  $ \alpha(\tau_e(|x|)e-\sum_{i=1}^n c_i u_i xu_i^*) \le \epsilon\alpha(e)$.  Thus, $$\begin{aligned} \|x\|_1&=\tau(|x|)  = \tau(e)\tau_e(|x|)=\frac {\tau(e)}{\alpha(e)}\alpha(\tau_e(|x|)e) \\ &\le \frac {\tau(e)}{\alpha(e)} [\alpha(\tau_e(|x|)e-\sum_{i=1}^n c_i u_i xu_i^*) + \alpha(\sum_{i=1}^n c_i u_i xu_i^*)] \\
&\le \epsilon \tau(e)+ \frac {\tau(e)}{\alpha(e)} \sum_{i=1}^n\alpha(c_i u_i xu_i^*) \\
 &\le \epsilon \tau(e)+ \frac {\tau(e)}{\alpha(e)} \alpha(x). \end{aligned}$$ 
   Letting $\epsilon\rightarrow 0$, we find that $\tau(x)\leq \frac {\tau(e)}{\alpha(e)} \alpha(x)$ for every $x$ in $\mathcal{M}_e$.  Namely,
            \begin{equation}
            \|exe\|_1\leq c\alpha(exe)  \qquad \text{ for all } x\in \mathcal I. \label{3.1}
            \end{equation}
        where $c=\frac {\tau(e)}{\alpha(e)}$.  Thus, $\alpha$ is locally $\|\cdot\|_1$-dominating.

        We now show that $\alpha$ is mutually continuous with respect to $\tau$. Actually, we need only to show that, if $\{e_\lambda\}$ is a net of projections in $\mathcal I$ such that $\alpha (e_\lambda)\rightarrow 0$, then $\tau(e_\lambda)\rightarrow 0$. Assume, to the contrary, that there exist a positive number $\epsilon>0$ and  a family $\{e_n\}$ of projections in $\mathcal I$ such that $\alpha(e_n)<1/n$ but $\tau(e_n)>\epsilon$ for each $n\in\mathbb N$. As $\mathcal M$ is a semifinite factor and $\alpha$ is unitarily invariant, we might assume further that $\{e_n\}_n$ is a decreasing sequence of projections in $\mathcal I$. Let $e_0 =\wedge_n e_n$. Then $\tau(e_0)\ge \epsilon$ and $\alpha(e_0)=0$ as $e_0\le e_n$ implies $\alpha(e_0)\le \alpha(e_n)<1/n$ for each $n$. This is a contradiction. Therefore,  if $\{e_\lambda\}$ is a net of projections in $\mathcal I$ such that $\alpha (e_\lambda)\rightarrow 0$, then $\tau(e_\lambda)\rightarrow 0$.

    \end{proof}

    \subsubsection{Non-commutative Banach function spaces}

In this subsection, we follow the notation of de Pagter in \cite{depag}.
We suppose, as before, that $\mathcal{M}$ is a von Neumann algebra with a semifinite, faithful, normal tracial state $\tau$. In this case, we have the ideal of the distrubtion function $d_x$, where $x$ is a $\tau$-measurable operator in $\mathcal{M}$.  We define $d_x$ by
    $$d_x(\lambda)= \tau(e^{|x|}(\lambda, \infty))\text{ for every }\lambda \geq 0,$$ where $e^{|x|}(\lambda, \infty)$ is the spectral projection of $|x|$ on $(\lambda,\infty)$.
It is easy to see that $d_x$ is decreasing, right-continuous and $d_x(\lambda)\rightarrow 0$ as $\lambda\rightarrow \infty$.  This allows us to define a generalized singular value function $$\mu(x;t)=\inf\{\lambda\geq 0:d_x(\lambda)\leq t\}\text{ for a given }t\geq 0 \text{ and for every }x\in\mathcal{M}.$$

\begin{definition}
Suppose that $(X,\Sigma, \nu)$ is a localizable measure space with the finite subset property.  Let $E$ be a two-sided ideal of the set of all complex-valued, $\Sigma$-measurable functions on $X$ with the identification of all functions equal a.e. with respect to $\nu$.  If $E$ has a norm $\|\cdot\|_E$ such that $(E,\|\cdot\|_E)$ is a Banach lattice, then $E$ is called a \emph{Banach function space}.
\end{definition}
 We assume that $E$ is a symmetric Banach function space on $(0,\infty)$ with Lebesgue meausure (see definition 2.6 in \cite{depag}).

 Following \cite{depag}, we  let $\mathcal{I}=\{x\in \mathcal{M}:x \text{ is a finite rank operator in }(\mathcal{M},\tau), \text{ and }\|\mu(x)\|_{E}<\infty\}$ and define a Banach function space $\mathcal{I}(\tau)$ equipped with a norm $\|\cdot\|_{\mathcal{I}(\tau)}$ such that
 $$\|x\|_{\mathcal{I}(\tau)}= \|\mu(x)\|_{E} \text { for every }x\in \mathcal{I}.$$

 Denote the closure of $\mathcal{I}$ under $\|\cdot\|_{\mathcal{I}(\tau)}$ by $\mathcal{I}(\tau)$. We will use the following Lemma to show that the restriction of $\|\cdot\|_{\mathcal{I}(\tau)}$ on $\mathcal I$ is a unitarily invariant, locally $\|\cdot\|_1$-dominating, mutually continuous norm with respect to $\tau$.

 \begin{lemma} \label{lemma3.7}
    Suppose that $y_0$ is an element of $\mathcal{I}$ such that $y_0=\sum_{i=1}^{n} \beta_i p_i$ where $\beta_1, \beta_2,\ldots, \beta_n $ are nonnegative and $p_1,\ldots, p_n$ are projections in $\mathcal{M}$  such that  $\tau(p_1)=\tau(p_2)=\dots=\tau(p_n)$.  Then
        $$\|y_0\|_{\mathcal{I}(\tau)}\geq \frac{\|p_1+\dots+p_n\|_{\mathcal{I}(\tau)}}{\tau(p_1+\dots+p_n)}\|y_0\|_1.$$
 \end{lemma}

 \begin{proof}
 Note that  $y_0$ is an element of $\mathcal{I}$ such that $y_0=\sum_{i=1}^{n}\beta_i p_i$ where $\tau(p_1)=\tau(p_2)=\dots=\tau(p_n)$.  Now let $\beta_{n+j}=\beta_j$ for all $1\le j\le n$ and $y_j=\sum_{i=1}^{n} \beta_{i+j} p_i$ for $1\leq j \leq n$.  Then, by definition, $\sum_{k=1}^{n} y_k=(\beta_1+\dots+\beta_n)(p_1+\dots +p_n)$, and also $\| y_k\|_{\mathcal{I}(\tau)}=\| y_0\|_{\mathcal{I}(\tau)}$ for all $1\le k\le n$. Therefore,
    \begin{align}
        \|y_0\|_{\mathcal{I}(\tau)}&\ge\frac{\|\sum_{k=1}^{n} y_k\|_{\mathcal{I}(\tau)}}{n}\notag\\
        &\geq (\frac{\beta_1+\dots+\beta_n}{n})\|p_1+\dots+p_n\|_{\mathcal{I}(\tau)} \notag\\
        &=\frac{\tau(y_0)}{\tau(p_1+\dots+p_n)}\|p_1+\dots +p_n\|_{\mathcal{I}(\tau)} \notag \\
        &=\|y_0\|_1\frac{\|p_1+\dots+p_n\|_{\mathcal{I}(\tau)}}{\tau(p_1+\dots+p_n)}\notag.
    \end{align}
 \end{proof}

 \begin{proposition} \label{prop3.3}
Suppose that $\mathcal{I}(\tau)$ is a Banach function space. Suppose that $\mathcal M$ is a diffuse von Neumann algebra   with a semifinite, faithful, normal tracial state $\tau$ and with an order continuous norm $\|\cdot\|_{\mathcal{I}(\tau)}$.  Then the restriction of $\|\cdot\|_{\mathcal{I}(\tau)}$ on $\mathcal I$   is a unitarily invariant, locally $\|\cdot\|_1$-dominating, mututally continuous norm with respect to $\tau$.
 \end{proposition}

\begin{proof}  Note $\|\cdot\|_{\mathcal{I}(\tau)}:\mathcal{I}\rightarrow [0,\infty)$ is a norm.  Now we will verify that  $\|\cdot\|_{\mathcal{I}(\tau)}$ satisfies the following conditions:
    \begin{enumerate}
        \item $\|uxv\|_{\mathcal{I}(\tau)}=\|x\|_{\mathcal{I}(\tau)}$ for all unitaries $u,v$ in $\mathcal{M}$, and every $x$ in $\mathcal{I}$;
        \item for every projection $e$ in $\mathcal{M}$ with $\tau(e)<\infty$, there exists $c(e)<\infty$ such that $\|exe\|_{\mathcal{I}(\tau)}\geq c(e)\|exe\|_1$ for all $x\in\mathcal M$;
        \item
        \begin{enumerate}
            \item[a.] if $\{e_\lambda\}_{\lambda\in\Lambda}$ is a net in $\mathcal{M}$ such that $e_\lambda\rightarrow I$ in the weak* topology, then $\|e_\lambda x-x\|_{\mathcal{I}(\tau)}\rightarrow 0$ for every $x\in \mathcal{I}$.
            \item[b.] if $\{e_\lambda\}_{\lambda\in\Lambda}$ is a net in $\mathcal{M}$ such that $\|e_\lambda\|_{\mathcal{I}(\tau)}\rightarrow 0$, then $\tau(e_\lambda)\rightarrow 0$.
        \end{enumerate}
    \end{enumerate}

(1) We begin by showing that $\|uxv\|_{\mathcal{I}(\tau)}=\|x\|_{\mathcal{I}(\tau)}$.

Given any $x$ and $y$ in $\mathcal{I}$,  we know that if $\tau(|x|^n)=\tau(|y|^n)$ for every $n\in\mathbb{N}$, then $\|x\|_{\mathcal{I}(\tau)}=\|y\|_{\mathcal{I}(\tau)}$ from definition 3.4 in \cite{depag}.  We have that $\tau$ is unitarily invariant by defintion, so for all unitaries $u$ and $v$ in $\mathcal{M}$ and $x$ in $\mathcal{I}$,
    $$\tau(|uxv|^n)=\tau(v^{-n}|x|^nv^n)=\tau(|x|^n)\text{ for every }n\in\mathbb{N}.$$
Hence $\|uxv\|_{\mathcal{I}(\tau)}=\|x\|_{\mathcal{I}(\tau)}$, and $\|\cdot\|_{\mathcal{I}(\tau)}$ is unitarily invariant.


(3) a. We show that if $\{e_\lambda\}\subseteq \mathcal{I}$ is an increasing net of projections such that $e_\lambda \rightarrow I$ in the weak* topology, then $e_\lambda x \rightarrow x$ in $\|\cdot\|_{\mathcal{I}(\tau)}$-norm for each $x\in\mathcal I$.

Suppose that $\{e_\lambda\}\subseteq \mathcal{I}$ is an increasing net of projections such that $e_\lambda \rightarrow I$ in the weak* topology.  By definition, $\|\cdot\|_{\mathcal{I}(\tau)}$ is order continuous. So for every $x$ in $\mathcal{I}$, $\|\sqrt{x^*(I-e_\lambda)x}\|_{\mathcal{I}(\tau)}\rightarrow 0$, and  $\|(I-e_\lambda)x\|_{\mathcal{I}(\tau)}=\| |(I-e_\lambda)x| \|_{\mathcal{I}(\tau)}=\|\sqrt{x^*(I-e_\lambda)x}\|_{\mathcal{I}(\tau)}$ by (1).  Therefore, $\|x-e_\lambda x\|_{\mathcal{I}(\tau)} \rightarrow 0$ for every $x$ in $\mathcal{I}$, as desired.


b. We show that if $\{e_\lambda\}\subseteq \mathcal{I}$ is a  net of projections such that $\|e_\lambda\|_{\mathcal{I}(\tau)}\rightarrow 0$, then $\tau(e_\lambda)\rightarrow 0$.

  We suppose that $\{e_\lambda\}\subseteq \mathcal{I}$ is a  net of projections such that $\|e_\lambda\|_{\mathcal{I}(\tau)}\rightarrow 0$.  Suppose to the contrary, that $\tau(e_\lambda)\nrightarrow 0$.  There exist an $\epsilon_0>0$, a subsequence $\{e_{\lambda_n}\}$ of $\{e_\lambda\}_{\lambda\in\Lambda}$  such that for every $n\ge 1$, $\tau(e_{\lambda_n})\geq \epsilon_0$.  As $\|e_\lambda\|_{\mathcal{I}(\tau)}\rightarrow 0$, $\|e_{\lambda_n}\|_{\mathcal{I}(\tau)}\rightarrow 0$.  Recall that $\mathcal{M}$ has no minimal projection.  By the properties of the norm $\|\cdot \|_{\mathcal{I}(\tau)}$, we might assume that  $\{e_{\lambda_n}\}$ is a decreasing sequence of projections in $\mathcal I$.    Thus there exist an $x=\wedge_n e_{\lambda_n}$ in $\mathcal{M}$ such that $0\le x\leq e_{\lambda_n}$ for every n, and $\epsilon_0\le \tau(x) \leq \tau(e_{\lambda_n})$.  Moreover, we have that $\|e_{\lambda_n}\|_{\mathcal{I}(\tau)}\geq \|x\|_{\mathcal{I}(\tau)}$ for every n, so therefore, $\|x\|_{\mathcal{I}(\tau)}=0$.  Hence $x=0$, which   contradicts with the fact that $\epsilon_0\le \tau(x)$.


(2) We show that for a projection $e\in\mathcal{M}$ such that $\tau(e)<\infty$ there exists $c(e)=\frac{\|e\|_{\mathcal{I}(\tau)} }{\tau(e)}$ satisfying $\|exe\|_{\mathcal{I}(\tau)}\geq c(e)\|exe\|_1$ for all $x\in \mathcal M$.

Suppose that $e=e^2=e^*$ is a projection in $\mathcal{M}$ such that $\tau(e)<\infty$.  Let $x$ be a positive element in $\mathcal M$. For any $\epsilon>0$, there exist nonnegative numbers $ \beta_1,\beta_2, \dots, \beta_n $ and subprojections $p_1, p_2, \dots, p_n$ of $e$ in $\mathcal M$ such that $\|exe-\sum_{i=1}^{n} \beta_i p_i\|_{\mathcal{I}(\tau)}\le \|e-\sum_{i=1}^{n} \beta_i p_i \|\|e\|_{\mathcal{I}(\tau)}<\epsilon$ and $\|exe-\sum_{i=1}^{n} \beta_i p_i\|_1\le \|e-\sum_{i=1}^{n} \beta_i p_i \|\|e\|_1<\epsilon$.  We call $\sum_{i=1}^{n} \beta_i p_i=y_0$.  For each  $m\in\mathbb{N}$ and  $1\leq i\leq n$ , we partition $p_i=q_{i,1}+q_{i,2}+\dots+q_{i,k_i}+q_{i, k_i+1}$ where $k_{i}$ is a positive integer and $q_{i,1}, q_{i,2}, \dots, q_{i,k_i}$ are projections in $\mathcal M$  such that $\tau(q_{i,1})=\tau(q_{i,2})=\dots=\tau(q_{i,k_i}) =1/m$, and $0\le\tau(q_{i,k_i+1})<1/m$.  We can write
    $$y_0=\sum_{i=1}^{n}\beta_i (\sum_{j=1}^{k_i+1} q_{i,j})=z_1+z_2,$$
where $z_1=\sum_{i=1}^{n} \beta_i (\sum_{j=1}^{k_i} q_{i,j})$ and $z_2=\sum_{i=1}^{n} \beta_i q_{k_i+1}$.

We let $q=\sum_{i=1}^{n}\sum_{j=1}^{k_i}q_{i,j}$.  Then,  by Lemma \ref{lemma3.7}, $$\|y_0\|_{ \mathcal{I}{(\tau)}}\geq\|z_1\|_{\mathcal{I}(\tau)}\geq \frac{\|q\|_{\mathcal{I}(\tau)}}{\tau(q)}\|z_1\|_1.$$  Also, by the triangle inequality,
    $$\|z_1\|\geq \|y_0\|_1-\|z_2\|_1\geq \|y_0\|_1-(\sum_{i=1}^{n} \beta_i) /m,$$
which approaches $\|y_0\|_1$ as $m\rightarrow \infty$. Furthermore, by (3) we have
$$  \frac{\|q\|_{\mathcal{I}(\tau)}}{\tau(q)} \ge  \frac{\|e\|_{\mathcal{I}(\tau)}-\sum_{i=1}^{n} \beta_i \|q_{i,k_i+1}\|_{\mathcal{I}(\tau)}}{\tau(e)} \rightarrow  \frac{\|e\|_{\mathcal{I}(\tau)} }{\tau(e)}\ \text{ as $m\rightarrow \infty$}.
$$ Therefore,
        $$\|y_0\|_{\mathcal{I}{(\tau)}}\geq \frac{\|e\|_{\mathcal{I}(\tau)} }{\tau(e)}\|y_0\|_1.$$
By the choice of $y_0$, we conclude that
$$\|exe\|_{\mathcal{I}{(\tau)}}\geq \frac{\|e\|_{\mathcal{I}(\tau)} }{\tau(e)}\|exe\|_1, $$
for all $x$ in $\mathcal M$.

    \end{proof}


\subsection{Embedding from $L^\alpha(\mathcal{M},\tau)$ into $\widetilde{M}$}

    We would like to show that there is a natural embedding from $L^\alpha (\mathcal{M},\tau)$ into $\widetilde{\mathcal{M}}$.

    Suppose that $\mathcal{M}$ is a von Neumann algebra with a semifinite, faithful, normal tracial weight $\tau$, and $\mathcal{H}$ is a Hilbert space. Recall
 $$\mathcal{I}=span\{xey: x,y\in\mathcal{M}, e\in\mathcal{M}, e=e^2=e^* \text{ with }\tau(e)<\infty\}$$ is the set of elementary operators of $\mathcal{M}$.
 Define $\widetilde{\mathcal{M}}$ to be the algebra of closed, densely defined operators on $\mathcal{H}$ affiliated with $\mathcal{M}$. We recall that the measure topology on $\mathcal{M}$ is given by the family of neighborhoods $U_{\delta,\epsilon}=\{a\in\mathcal{M}:  \|ap\|\leq \epsilon \text{ and } \tau(p^\perp)\leq \delta\text{ for some projection }p \in\mathcal{M}\} $ for any $\epsilon,\delta>0$.

Suppose that  $\alpha$ is a unitarily invariant, locally $\|\cdot\|_1$-dominating, mutually continuous norm with respect to $\tau$ on $\mathcal{M}$.

    \begin{lemma}\label{lemma3.10}
        Let $\epsilon>0$ be given.  There exists $\delta_0>0$ such that if $e$ is a projection in $\mathcal{I}$ with $\alpha(e)<\delta_0$, then $\tau(e)<\epsilon$.
    \end{lemma}

    \begin{proof}
        Suppose, to the contrary, that there exists an $\epsilon>0$ such that for every $\delta_0>0$, there exists a projection $e_{\delta_0}$ in $\mathcal I$ such that $\alpha(e_{\delta_0})<\delta$, and $\tau(e_{\delta_0})\geq\epsilon$.  Let $\delta_0=1/n$ for each $n\in\mathbb N$.  Then there exits a sequence $\{e_n\}_{n\in\mathbb{N}}$ such that for every $n\in\mathbb{N}$, $\alpha(e_n)<1/n$, and $\tau(e_n)\geq\epsilon$.  This is a contradiction, as $\alpha$ is mutually continuous with respect to $\tau$ (see definition \ref{def3.1}).  Therefore, the Lemma is proven.
    \end{proof}

    \begin{lemma}\label{lemma2.3}
        Suppose a sequence $\{a_n\}$ in $\mathcal{I}$ is Cauchy with respect to the norm $\alpha$.  Then $\{a_n\}$ is Cauchy in the measure topology.
    \end{lemma}

    \begin{proof}
        To prove that $\{a_n\}\subseteq \mathcal I$ is Cauchy in the measure topology, it suffices to show that for every $\epsilon, \delta>0$, there exists an $N\in\mathbb{N}$ such that for $n,m>N$, there exists a projection $p_{m,n}$ satisfying $\| |a_m-a_n|p_{m,n}\|< \delta$ and $\tau((p_{m,n})^\perp)<\epsilon$.  By Lemma \ref{lemma3.10}, we know that there exists a $\delta_0>0$  such that
        \begin{equation}
        \text{if } e \text{ is a projection in  }\mathcal{I} \text{ with }\alpha(e)<\delta_0\text{, then }\tau(e)<\epsilon. \label{**}
        \end{equation}
        For each $m,n\in\mathbb{N}$, let $\{e_\lambda(m,n)\}$ be the spectral decomposition of $|a_m-a_n|$ in $\mathcal M$.  By the spectral decomposition theorem, we have $|a_m-a_n|=\int_0^{\infty}\lambda de_\lambda(m,n)$, and $\tau(|a_m-a_n|)=\int_0^{\infty} \lambda d\tau (e_\lambda(m,n))$.  Let $\lambda_0=\delta_0$.  Hence $\lambda_0 e_{\lambda_0}(m,n)^\perp \leq |a_m-a_n|e_{\lambda_0}(m,n)^\perp$.  So
            \begin{equation}
            \alpha(\lambda_0 e_{\lambda_0}(m,n)^\perp)\leq \alpha(|a_m-a_n|) \text{ for all } m,n\in\mathbb{N}.  \label{*}
            \end{equation}
        Recall that $\{a_n\}$ is Cauchy in $\alpha$-norm.  For $\epsilon_1=\lambda_0\delta_0 >0$, there exists $N\in\mathbb{N}$ such that for all $m,n>N$, $\alpha(a_m-a_n)<\epsilon_1$.  Combining with (\ref{*}), we have that for every $m,n>N$, $\lambda_0 \alpha(e_{\lambda_0}(m,n)^\perp)<\epsilon_1$. This implies that
            $$\alpha(e_{\lambda_0}(m,n)^\perp)<\epsilon_1/\lambda_0=\delta_0.$$

        Because of (\ref{**}), $\tau(e_{\lambda_0}(m,n)^\perp)<\epsilon$ for every $m,n>N$. Put $p_{m,n}=e_{\lambda_0}(m,n)$.  Then for every $m,n>N$,
            $$\||a_m-a_n|p_{m,n}\|\leq \lambda_0=\delta_0, \text{ and }\tau(p_{m,n}^\perp)<\epsilon.$$
        The proof is complete.
    \end{proof}

Therefore, there is a natural continuous mapping from $L^\alpha(\mathcal{M},\tau)$ into $\widetilde{\mathcal{M}}$.

    Let $e$ be a projection in $\mathcal{M}$ such that $\tau(e)<\infty$, and let $\mathcal{M}_e=e\mathcal{M}e$.  Define a faithful, normal, tracial state $\tau_e$ on $\mathcal{M}_e$ by $\tau_e(x)=\frac{1}{\tau(e)}\tau(x)$ for every $x$ in $\mathcal{M}_e$.

    It can be shown that $\tau_e$ is a finite, faithful, normal tracial state on $\mathcal{M}_e$.  Suppose that $\alpha$ is a unitarily invariant, locally $\|\cdot\|_1$-dominating, mutually continuous norm with respect to $\tau$ on $\mathcal{M}$.   Define $\alpha_{e}=\alpha|_{e\mathcal{M}e}$.  We define $\alpha_{e}':\mathcal{M}_e\rightarrow [0,\infty]$ by $\alpha_{e}'(x)=\sup\{|\tau(xy)|:y\in\mathcal{M}, \alpha_{e}(y)\leq 1\}$ for every $x$ in $\mathcal{M}_e$.  It may be shown that $\alpha_{e}'$ is indeed a norm, and we call $\alpha_{e}'$ the \emph{dual norm} of $\alpha_{e}$ (see \cite{CHS} for more information).  We define $L^{\alpha_{e}'}(\mathcal{M}_e,\tau)=\overline{\mathcal{M}_e}^{\alpha_{e}'}$.

    We may also define $\overline{\alpha_{e}}:L^1(\mathcal{M}_e,\tau)\rightarrow [0,\infty]$ by $\overline{\alpha_{e}}(x)=\sup\{|\tau(xy)|:y\in\mathcal{M}, \alpha_{e}'(y)\leq 1\}$ for every $x$ in $\mathcal{M}_e$, and $\overline{\alpha_{e}'}:L^1(\mathcal{M}_e,\tau)\rightarrow [0,\infty]$ by $\overline{\alpha_{e}'}=\sup\{|\tau(xy)|:y\in\mathcal{M}, \alpha_{e}(y)\leq 1\}$ for every $x$ in $\mathcal{M}_e$.  $L^{\overline{\alpha_{e}}}(\mathcal{M}_e,\tau)$ and $L^{\overline{\alpha_{e}'}}(\mathcal{M}_e,\tau)$ are defined to be $\overline{\mathcal{M}_e}^{\overline{\alpha_{e}}}$ and $\overline{\mathcal{M}_e}^{\overline{\alpha_{e}}'}$ respectively.



\begin{lemma}\label{lemma2.3.1}
    Let $\alpha$ be a unitarily invariant, $\|\cdot\|_1$-dominating, mutually continuous norm with respect to $\tau$.  Then $\alpha_{e}$, $\alpha_{e}'$, $\overline{\alpha_{e}'}$ and $\overline{\alpha_{e}}$ are unitarily invariant norms on $L^\alpha(\mathcal{M},\tau)$.
\end{lemma}

\begin{proof}
    Clearly, $\alpha_{e}(uxv)=\alpha(uxv)=\alpha(x)=\alpha_{e}(x)$ for unitaries $u$ and $v$ and an element $x$ in $\mathcal{M}_{e}\subset\mathcal{M}$.  Therefore, $\alpha_{e}$ is a unitarily invariant norm.

    Let $u$ and $v$ be unitaries, and $x$ be an element of $L^{\alpha_{e}'}(\mathcal{M}_{e},\tau_e)$.  Then
    \begin{align}
     \alpha_{e}'(uxv)& =\sup\{|\tau(uxvy)|:y\in\mathcal{M},\alpha_{e}(y)\leq 1\} \notag\\
        & =\sup\{|\tau(xuyv)|:y\in\mathcal{M},\alpha_{e}(y)\leq 1\}\notag\\
        &=\sup\{|\tau(xy_0)|:y_0\in\mathcal{M}, \alpha_{e}(y_0)\leq 1\}\notag\\
        &=\alpha_{e}'(x) \notag
     \end{align}
     for every $x\in L^{\alpha_{e}'}(\mathcal{M}_{e},\tau_e)$.
    Therefore, $\alpha_{e}'$ is unitarily invariant.

    The proofs that $\overline{\alpha_{e}}$ and $\overline{\alpha_{e}'}$ are unitarily invariant are similar.
\end{proof}

\begin{lemma}\label{lemma3.11.2}
Suppose $\alpha$ is a unitarily invariant, locally $\|\cdot\|_1$-dominating, mutually continuous norm with respect to $\tau$ on $\mathcal{M}$.  Then
    \begin{enumerate}
        \item[(i)] $\|x\|_1\leq \overline{\alpha_e}(x)$ for every $x\in L^{\overline{\alpha_e}}(\mathcal{M}_e,\tau)$; and
        \item[(ii)] $\|x\|_1\leq \overline{\alpha_e}'(x)$ for every $x\in L^{\overline{\alpha_e}'}(\mathcal{M}_e,\tau)$.
        \end{enumerate}
\end{lemma}

\begin{proof}

    (i) Suppose that $x$ is in $L^{\overline{\alpha_e}}(\mathcal{M}_e,\tau)\subseteq L^1(\mathcal{M}_e,\tau)$.  Let $x=uh$ be the polar decomposition of $x$ in $L^1(\mathcal{M}_e,\tau)$, such that $u$ is a unitary in $\mathcal{M}_e$, and $h$ is positive in $L^1(\mathcal{M}_e,\tau)$.  As $\overline{\alpha_e}$ is unitarily invariant (see Lemma \ref{lemma2.3.1}),
        \begin{equation}
        \overline{\alpha_e}(x)=\overline{\alpha_e}(uh)=\overline{\alpha_e}(h).  \label{3.12}
        \end{equation}
    By definition, $\overline{\alpha_e}(h)\geq|\tau(h)|=\|x\|_1$. Hence, combining with \ref{3.12},
        $$\|x\|_1\leq \overline{\alpha_e}(x).$$

    (ii) The proof of (ii) is similar.
\end{proof}
\begin{lemma}\label{lemma2.4}
        For every $y\in \mathcal{M}_e$ and every $z\in L^1(\mathcal{M}_e,\tau)$, $\alpha_e '(yz)\leq \|y\|\overline{\alpha_e '}(z)$.
    \end{lemma}

    \begin{proof}
        Suppose $y\in \mathcal{M}_e$ such that $\|y\|=1$, and let $y=\omega |y|$ be the polar decomposition of $y$ in $\mathcal{M}_e$, i.e. $\omega \in \mathcal{M}_e$ is unitary and $|y|\in \mathcal{M}_e$ is positive.  Define $v=|y|+i\sqrt{1-|y|^2}$. Then by construction, $v$ is unitary in $\mathcal{M}_e$, and $|y|=\frac{v+v^*}{2}$.  Consider any $z$ in $L^1(\mathcal{M}_e,\tau)$.  Then we have that
            $$\overline{\alpha_e '}(yz)=\overline{\alpha_e '}(\omega|y|z)=\overline{\alpha_e '}(\frac{vz+v^*z}{2})\leq \frac{\overline{\alpha_e '}(vz)+\overline{\alpha_e '}(v^*z)}{2}\text{ for every }z\text{ in }L^1(\mathcal{M}_e,\tau),$$
        and $y$ in $\mathcal{M}_e$ such that $\|y\|=1$.  Thus $\overline{\alpha_e '}(yz)\leq \|y\|\overline{\alpha_e '}(z)$ for every $z$ in $L^1(\mathcal{M}_e,\tau)$ and $y$ in $\mathcal{M}_e$.
    \end{proof}

    \begin{lemma} \label{lemma2.2}
        For every $x\in\mathcal{M}_e$, $\alpha_{e}(x)=\overline{\alpha_{e}}(x)$.
    \end{lemma}

    \begin{proof}
        First, we show that $\overline{\alpha_{e}}(x)\leq \alpha_{e}(x)$ for every $x$ in $\mathcal{M}_e$.  By definition, $|\tau(xy)|\leq \alpha_{e}(x)\alpha_{e}'(y)$ for every $x$ and $y$ in $\mathcal{M}_e$.  Suppose $\alpha_{e}'(y)\leq 1$.  Then $|\tau(xy)|\leq\alpha_{e}(x)\alpha_{e}'(y)<\alpha_{e}(x)$ for every $x$ in $\mathcal{M}_e$, and $y$ in $\mathcal{M}_{e}$ such that $\alpha_{e}'(y)\leq 1$. Hence
            \begin{equation}
            \overline{\alpha_{e}}(x)=\sup\{|\tau(xy)|:y\in \mathcal{M}_e, \alpha_{e}'(y)\leq 1\}\le \alpha_{e}(x) \label{3.13}
            \end{equation}
        by definition.

        Next, we show that $\overline{\alpha_{e}}(x)\geq \alpha_{e}(x)$.  Suppose $x$ is in $\mathcal{M}_e$ with $\alpha_{e}(x)=1$.  Then by the Hahn-Banach Theorem, there exists a $\varphi$ in $L^{\alpha_{e}}(\mathcal{M}_{e},\tau)^\#$ such that $\varphi(x)=\alpha_{e}(x)=1$, and $\|\varphi\|=1$.  Since $\varphi$ is in $L^{\alpha_{e}}(\mathcal{M}_{e},\tau)^\#$, there exists $\xi$ in $L^{\overline{\alpha_{e}'}}(\mathcal{M}_e,\tau)$ such that $\varphi(x)=|\tau(x\xi)|=1$, and $\overline{\alpha_{e}}'(\xi)=\|\xi\|=1$.  Let $\xi=uh$ be the polar decomposition of $\xi$ in $L^{\overline{\alpha_{e}'}}(\mathcal{M}_e,\tau)$, where $u\in \mathcal{M}_e$ is unitary and $h\in L^{\overline{\alpha_{e}'}}(\mathcal{M}_e,\tau)$ is positive.

        By Lemma 3.8 in \cite{CHS}, there exists a family $\{e_\lambda\}$ of projections in $\mathcal{M}_e$ such that $\|h-he_\lambda\|_1\rightarrow 0$, and $e_\lambda h=he_\lambda\in \mathcal{M}_e$ for every $0<\lambda<\infty$.  Also, $u\in \mathcal{M}_e$, so $uhe_\lambda\in \mathcal{M}_e$.  Thus $\alpha_{e}'(uhe_\lambda)=\overline{\alpha_{e}}'(uhe_\lambda)\leq \overline{\alpha_{e}}'(uh)\|e_\lambda\|\leq \overline{\alpha_{e}}'(uh)=\alpha_{e}'(\xi)=1$, as $\alpha_{e}'(x)=\overline{\alpha_{e}}'(x)$ for every $x\in \mathcal{M}_e$ by Lemma 3.2 in \cite{CHS}.  So, $\alpha_{e}(x)|\tau(x\xi)|=|\tau(xuh)|=\lim_{\lambda\rightarrow\infty}|\tau(xuhe_\lambda)|\leq\sup\{|\tau(xy)|:y\in \mathcal{M}_e, \alpha_{e}'(y)\leq 1\}=\overline{\alpha_{e}}(x)$.  Therefore
            \begin{equation}
            \alpha_{e}(x)\leq \overline{\alpha_{e}}(x).\label{3.14}
            \end{equation}
        Hence from equations \ref{3.13} and \ref{3.14}, $\alpha_{e}(x)=\overline{\alpha_{e}}(x)$, and the Lemma is proven.
    \end{proof}

    \begin{lemma}
        $L^{\overline{\alpha_e}}(\mathcal{M}_e,\tau)=\{x\in L^1(\mathcal{M}_e):\overline{\alpha_e}(x)<\infty\}$ is a complete space in $\alpha_e$-norm.
    \end{lemma}

    \begin{proof}
        It suffices to show that for every Cauchy sequence $\{b_n\}$ in $L^{\overline{\alpha_e}}(\mathcal{M}_e,\tau)$, there exists $b$ in $L^{\overline{\alpha_e}}(\mathcal{M}_e,\tau)$ such that $b_n\rightarrow b$ in $\overline{\alpha_e}$-norm. Suppose that $\{b_n\}$ is a Cauchy sequence in $L^{\overline{\alpha_e}}(\mathcal{M}_e,\tau)$.  There exists $M>0$ such that $\overline{\alpha_e}(b_n)\leq M$ for every $n$.

        By Lemma \ref{lemma3.11.2},
            $$\|b_n-b_m\|_1\leq \overline{\alpha}(b_n-b_m)\text{ for all }m,n\geq 1.$$
        Therefore, $\{b_n\}$ is Cauchy in $L^1(\mathcal{M}_e,\tau)$, which is complete.  So there exists a $b_0$ in $L^1(\mathcal{M}_e,\tau)$ such that $\|b_n-b_0\|_1\rightarrow 0$.

        First, we claim that $b_0$ is in $L^{\overline{\alpha_e}}(\mathcal{M}_e,\tau)$.  Let $y\in \mathcal{M}_e$ such that $\alpha_e'(y)\leq 1$.  We have that $|\tau(b_n y)-\tau(b_0 y)|=|\tau((b_n-b_0)y)|\leq \|b_n-b_0\|_1\|y\|_\infty$ by H\"older's Inequality. However, $\|b_n-b_0\|_1\|y\|_\infty\rightarrow 0$.  Also, by the definition of $\overline{\alpha}$, we also have that $|\tau(b_0 y)|=\lim_{n\rightarrow\infty}|\tau(b_n y)|\leq \limsup_{n\rightarrow \infty} \overline{\alpha_e}(b_n)\alpha_e '(y)\leq M$.  Therefore, $\overline{\alpha}(b_x)\leq M$, and $b_0\in L^{\overline{\alpha_e}}(\mathcal{M}_e,\tau)$.

        Now, we show that $\overline{\alpha_e}(b_n-b_0)\rightarrow 0$.  We know that $\{b_n\}$ is Cauchy in $L^{\overline{\alpha}}(\mathcal{M}_e,\tau)$, so for every $n\geq 1$,
            \begin{align}
            |\tau((b_n-b_0)y)| &= \lim_{m\rightarrow \infty}|\tau((b_m-b_n)y)| \notag\\
            &\leq \limsup_{m\rightarrow\infty}\overline{\alpha_e}(b_n-b_m)\alpha_e '(y) \notag\\
            &\leq \limsup_{m\rightarrow \infty}\overline{\alpha}(b_m-b_n)\notag
            \end{align}
        Therefore, $\overline{\alpha_e}(b_n-b_0)\leq \limsup_{m\rightarrow \infty}(b_n-b_m)$ for every $n\geq 1$, and since $\{b_n\}$ is Cauchy in $L^{\overline{\alpha_e}}(\mathcal{M}_e,\tau)$, $$\overline{\alpha_e}(b_n-b_0)\rightarrow 0\text{ as }n\rightarrow\infty,$$
        and the Lemma is proven.

    \end{proof}

    Therefore $L^{\overline{\alpha_e}}(\mathcal{M}_e,\tau)$ is a Banach space with respect to $\overline{\alpha_e}$-norm.

    \begin{lemma}\label{lemma2.6}  Suppose  that $e\in\mathcal{M} $ is a projection such that $\tau(e)<\infty$.
        Suppose $\{ea_n e\}\subseteq\mathcal{I}$ is Cauchy in $\alpha$-norm, and $ea_n e$ converges in measure to 0. Then
        \begin{enumerate}
            \item[(i)]for every $\epsilon>0$, there exists a $\delta>0$ such that, if $q$ is a projection in $\mathcal{M}$ with $\tau(q)<\delta$, $|\tau(ea_n e q)|<\epsilon$ for every $n$;
            \item[(ii)] given $\delta>0$, $\epsilon>0$ and $N\in\mathbb{N}$, there exists $p_n$, a projection in $\mathcal{M}$, such that $\|ea_n ep_n\| \leq \epsilon$, and $\tau(p_n^\perp)<\delta$ for every $n\geq N$;
            \item[(iii)] for every projection $q$ in $\mathcal{I}$, $\tau(ea_n eq)\rightarrow 0$ as $n\rightarrow \infty$; and
            \item[(iv)] for every $b$ in $\mathcal{M}$, $\tau(ea_n eb)\rightarrow 0$ as $n\rightarrow \infty$.
        \end{enumerate}
    \end{lemma}

    \begin{proof}
        (i) Suppose that, as above, $e\in\mathcal{M} $ is a projection such that $\tau(e)<\infty$  and $\{ea_n e\}$ s a Cauchy sequence in $\alpha$-norm.  Let $\epsilon>0$ be given. By assumption, $\alpha$ is a locally $\|\cdot\|_1$-dominating norm, so there exists $c(e)$ such that $\alpha(exe)\geq c(e)\|exe\|_1$ for every $x\in\mathcal{M}$. Then, given $\frac{\epsilon}{2}c(e)$, there exists $N_0\in\mathbb{N}$ such that for all $n,m>N_0$,
            $$\alpha(ea_n e-ea_m e)\leq \frac{\epsilon}{2}c(e).$$

        Let $\delta=\min_{k\leq N_0}\{\frac{\epsilon}{2\|ea_k e\|_\infty}\}$.  Suppose $q$ is a projection in $\mathcal{M}$ such that $\tau(q)\leq \delta$.  Then for every $k\leq N_0$, $|\tau(ea_k eq)|\leq \|ea_k e\|  \|q\|_1$ by  H\"older's Inequality, and $\tau(q)=\|q\|_1\leq \delta$.  Hence $|\tau(ea_k eq)|\leq \|ea_k e\| \delta <\epsilon/2$ for all $k\leq N_0$ by our choice of $\delta$.

        For $k>N_0$,
        \begin{align}
            |\tau(ea_k eq)|&\leq |\tau((ea_k e-ea_{N_0}e)q)|+|\tau(ea_{N_0} eq)|\notag\\
            &\leq\|ea_k e-ea_{N_0}e\|_1\|q\|  +\|ea_{N_0} e\| \|q\|_1 \tag{by H\"older's Inequality}\\
            &\leq \frac{1}{c(e)}\alpha(ea_ke - ea_{N_0}e)\|q\| +\|ea_{N_0}e\|  \delta \tag{by Definition \ref{3.1}}\\
            &<\epsilon/2+\epsilon/2=\epsilon. \notag
        \end{align}
        Hence, (i) is proven.



        (ii) Suppose that $\{ea_n e\}$ is a Cauchy sequence in $\alpha$-norm and $ea_ne\rightarrow 0$ in measure.  Then, by the definition of convergence in measure, for any $\epsilon>0$, $\delta>0$ and $N\in\mathbb{N}$, there exists $p_n$ in $\mathcal{M}$ such that $\|ea_n ep_n\|<\epsilon$ and $\tau(p_n^{\perp})<\delta$ for every $n\geq N$.

        (iii) Suppose that $\{ea_n e\}$ is a Cauchy sequence in $\alpha$-norm such that $ea_n e\rightarrow 0$ in measure.  The by (i), given $\epsilon>0$ and a projection $q$ in $\mathcal{I}$, there exists a $\delta_1>0$ such that if $\tau(q')<\delta_1$, then $|\tau(ea_n eq')|<\epsilon/2$.  Let $\delta>0$ and $\epsilon_1=\frac{\epsilon}{2\tau(q)}$.  Then by (ii), there exists  $N\in\mathbb{N}$ such that $\|ea_n ep_n\|<\epsilon_1$, and $\tau(p_n^{\perp})<\delta$ for every $n\geq N$.  Thus, for $n\geq N$ and any projection $q\in \mathcal{I}$,
            \begin{equation} \tau(ea_n eq)=\tau(ea_n e(q-q\cap p_n))+\tau(ea_n e(q\cap p_n)).\label{3.4}\end{equation}

        However, $\tau(q-q\cap p_n)=\tau(q\cup p_n-p_n)\leq \tau(p_n^\perp)<\delta$.  Therefore,
            \begin{equation} |\tau(ea_n e(q-q\cap p_n))|<\epsilon/2.\label{3.5}
            \end{equation}
        from (i). Also,
            \begin{align}
                |\tau(ea_n e(q\cap p_n))|&=|\tau(ea_n ep_n(q\cap p_n))|\notag\\
                &\leq \|ea_n ep_n\| \|q\cap p_n\|_1\notag\\
                &\leq \epsilon_1 \tau(q\cap p_n) \notag\\
                &<\epsilon_1\tau(q)=\epsilon/2. \label{3.6}
            \end{align}
        Then from equations \ref{3.4}, \ref{3.5} and \ref{3.6}, $|\tau(ea_n eq)|<\epsilon$ for any given $\epsilon>0$.  Therefore, $\tau(ea_n e)\rightarrow 0$ for every $q\in\mathcal{M}$ such that $q$ is a projection and $\tau(q)<\infty$.

        (iv) Suppose that $\{ea_n e\}$ is a Caucy sequence in $\alpha$-norm. Then there exists $M>0$ such that $\tau(ea_n e)\leq\frac{\alpha(ea_n e)}{c(e)}<\frac{M}{c(e)}$. By considering $ebe$ instead, we might assume that  $b\in\mathcal{I}$.  By the spectral decomposition theorem, $b$ can be approximated by a finite linear combination of projections $q_i$ in $\mathcal{M}$, i.e. there exist $q_i \in\mathcal{I}$ such that $\|b-\sum_{i=1}^{n} q_i\|<\epsilon \frac{c(e)}{M}$ for any given $\epsilon>0$.  Therefore,
            \begin{align}
                |\tau(ea_n eb)-\tau(ea_n e\sum_{i=1}^{n} q_i)|&=|\tau(ea_n e(b-\sum_{i=1}^{n} q_i))|\notag \\
                &\leq \|\tau(ea_n e)\|_1\|b-\sum_{i=1}^{n} q_i\| \notag\\
                &\leq \frac{M}{c(e)} \epsilon \frac{c(e)}{M}<\epsilon \notag.
            \end{align}
        Therefore, the Lemma is proven.

    \end{proof}

    \begin{proposition}
        There exists a natural embedding from $L^\alpha(\mathcal{M},\tau)$ into $\widetilde{\mathcal{M}}$.
    \end{proposition}

    \begin{proof}
        By Lemma \ref{lemma2.3}, there exists a natural mapping from $L^\alpha(\mathcal{M},\tau)$ to $\widetilde{\mathcal{M}}$.

        It suffices to show that this mapping is an injection.  Suppose that $\{a_n\}\subseteq \mathcal I$ is a Cauchy sequence in $\alpha$-norm such that $x_n\rightarrow 0$ in measure.  As $L^\alpha(\mathcal{M},\tau)$ is complete, there exists $a\in L^\alpha(\mathcal{M},\tau)$ such that $a_n\rightarrow a$ in $\alpha$-norm.  Assume that $a\neq 0$.  There exists a projection $e$ in $\mathcal{M}$ such that $\tau(e)<\infty$ and $eae\neq 0$. Thus  $\{ea_n e\}$ is Cauchy in $\alpha_e$-norm, $ea_n e\rightarrow 0$ in measure and $ea_n e\rightarrow eae\neq 0$ in $\alpha_e$-norm.  By Lemma \ref{lemma2.6}, $\tau(ea_n eb)\rightarrow 0$ for any $b\in\mathcal{M}$. As,
$|\tau(ea_neb)-\tau(eaeb)|\le \alpha_e(ea_ne-eae)\alpha'_e(b)\rightarrow 0$, we have
 $$
 \tau(eaeb)=0 \qquad \text{ for all } b\in\mathcal I.
 $$
 On the other hand,  by Lemma \ref{lemma2.2} and  definition of $\overline{\alpha_e}$, since $eae\ne 0$, there exists some $b_0 \in \mathcal{M}_e$ such that $\alpha_e'(b_0)\le 1$ and $\tau(eaeb_0)>\frac {\alpha(ea e)} 2 $.
 This is a contradiction.  Therefore, $a=0$, and the mapping is an embedding.

    \end{proof}

\section{Arveson's Non-Commutative Hardy Space}

In this section, we will extend Arveson's classical definition of a non-commutative Hardy space to $L^\alpha(\mathcal{M},\tau)$.  We assume, as before, that $\mathcal{M}$ is a von Neumann algebra with a semifinite, faithful, normal tracial weight $\tau$, and we assume that $\mathcal{A}\subseteq\mathcal{M}$ is a weak*-closed unital subalgebra of $\mathcal{M}$.  We let $\mathcal{D}=\mathcal{A}\cap \mathcal{A}^*$, and assume that $\Phi:\mathcal{M}\rightarrow \mathcal{D}$ is a faithful, normal conditional expection. Let $$\mathcal{I}=span\{xey: x,y\in\mathcal{M}, e\in\mathcal{M}, e=e^2=e^* \text{ with }\tau(e)<\infty\}$$ be the set of elementary operators of $\mathcal{M}$.

\begin{definition}\label{def2.8}
 A weak*-closed unital subalgebra $\mathcal{A}$ of $\mathcal{M}$ is called a semifinite subdiagonal subalgebra, or a semifinite non-commutative Hardy space with respect to $(\mathcal{M},\tau)$, if:
 \begin{enumerate}
    \item The restriction $\tau|_{\mathcal{D}}$ of $\tau$ to $\mathcal{D}=\mathcal{A}\cap\mathcal{A}^*$ is semifinite.
    \item $\Phi(xy)=\Phi(x)\Phi(y)$ for every $x$ and $y$ in $\mathcal{A}$.
    \item $\mathcal{A}+\mathcal{A}^*$ is weak*-dense in $\mathcal{M}$.
    \item $\Phi$ is $\tau$-preserving (i.e. $\tau(\Phi(x))=\tau(x)$ for every positive operator $x\in \mathcal{M}$).
 \end{enumerate}
 We will, in this case, denote $\mathcal{A}$ by $H^\infty$.
 \end{definition}

\begin{definition}
 Let $\alpha:\mathcal{I}\rightarrow [0,\infty)$ be a unitarily invariant, locally $\|\cdot\|_1$-dominating, mutually continuous norm with respect to $\tau$.  We denote by $H^\alpha$ the closure $[\mathcal{A}\cap L^\alpha(\mathcal{M},\tau)]_{\alpha}$ in $\alpha$-norm.
\end{definition}

\begin{remark}  Considering the conditional expectation $\Phi:\mathcal{M}\rightarrow \mathcal{D}$ from Definition \ref{def2.8}, we have that $\Phi$ extends to a projection from $L^{1}(\mathcal{M},\tau)$ to $L^{1}(\mathcal{D},\tau)$.  We still denote such an extension by $\Phi$, and we have that
$$\Phi(axb)=a\Phi(x)b, \qquad \forall a,b\in\mathcal{D}, \, x\in L^{\alpha}(\mathcal{M},\tau).$$

\end{remark}

\begin{notation}
We denote $ker(\Phi)\cap H^\infty$ by $H^{\infty}_0$, and $ker(\Phi)\cap H^{\alpha}$ by $H^{\alpha}_0$.
\end{notation}

    \begin{lemma}\label{lemma2.1}
        Suppose that $\mathcal{M}$ is a  von Neumann algebra with a semifinite, faithful, normal tracial weight $\tau$.  Let $\mathcal{A}=H^\infty$ be a semifinite subdiagonal subalgebra, as described in Definition \ref{def2.8}.  Let $e=e^*=e^2\in \mathcal{D}$ such that $\tau(e)<\infty$.  Then $eH^\infty e$, denoted $H^\infty_e$, is a Hardy space of $\mathcal{M}_e$.
    \end{lemma}

    \begin{proof}
        See Lemma 3.1 of \cite{Be}.
    \end{proof}

    \begin{lemma}\label{lemma2.12}
        Suppose $\mathcal{M}$ is a semifinite von Neumann algebra with a semifinite, faithful, normal tracial weight $\tau$.  Let $H^\infty$ be a semifinite, subdiagonal subalgebra of $\mathcal{M}$, as described in Definition \ref{def2.8}, namely that the restriction of $\tau$ to  $\mathcal{D}=H^\infty\cap (H^\infty)^*$ is semifinite.  Let $\alpha:\mathcal{I}\rightarrow [0,\infty)$ be a unitarily invariant, locally $\|\cdot\|_1$-dominating, mutually continuous norm with respect to $\tau$.  

        Then for every $x\in L^\alpha (\mathcal{M},\tau)$  and for every $e\in\mathcal{D}$ such that $\tau(e)<\infty$, there exist $h_1, h_3 \in e H^\infty e  =H^\infty_e$ and $h_2, h_4 \in eH^\alpha e=H^\alpha_e$ such that
        \begin{enumerate}
            \item[(i)] $h_1 h_2=e=h_2 h_1$ and $h_3 h_4=e=h_4 h_3$
            \item[(ii)] $h_1 ex\in\mathcal{M}$, and $exh_3 \in \mathcal{M}$.
        \end{enumerate}
    \end{lemma}

    \begin{proof}
        Let $ex=\sqrt{exx^*e}u=|x^* e|u$ be the polar decomposition of $(ex)^*$ in $L^\alpha(\mathcal{M},\tau)$ where $u$ is a partial isometry in $\mathcal{M}$ and $|x^* e|$ is a positive operator in $L^\alpha (\mathcal{M},\tau)$. Note that $|x^*e|$ is in $eL^\alpha(\mathcal{M},\tau)e=L^\alpha (\mathcal{M}_e, \tau)$.  Since $0<\tau(e)<\infty$, we know that $\mathcal{M}_e$ is a finite von Neumann algebra with a faithful, normal tracial state $\frac{1}{\tau(e)}\tau$.
        By Lemma \ref{lemma2.1}, we have that $H^\infty_e$ is a finite subdiagonal subalgebra of $\mathcal{M}_e$ with $[H^\infty_e]_\alpha = H^\alpha_e$.

        We have that $|x^* e| \in L^\alpha (\mathcal{M}_e, \frac{1}{\tau(e)}\tau)$, and $0<\tau(e)<\infty$.  Then $w=(e+|x^*e|)^{-1}$ is an invertible operator in $\mathcal{M}_e$ with $w^{-1}\in L^\alpha (\mathcal{M}_e, \frac{1}{\tau(e)}\tau)$.  We know that $\mathcal{M}_e$ is a finite von Neumann algebra with faithful, normal tracial state $\frac{1}{\tau(e)}\tau$, and $\alpha_e$ on $\mathcal{M}_e$ is a unitarily invariant, $\epsilon\text{-}\|\cdot\|_1$-dominating, continuous norm on $\mathcal{M}_e$.  Therefore, from Proposition 5.2 in \cite{CHS}, there exists a unitary $v$ in $\mathcal{M}_e$, $h_1\in H^\infty_e$, and $h_2\in H^\alpha_e$ such that
            \begin{enumerate}
                \item[(i)] $h_1 h_2=e=h_2 h_1$; and
                \item[(ii$_a$)] $w=vh_1$.
            \end{enumerate}
        By (ii$_a$), we get (ii$_b$) $h_1 |x^* e|=v^* w|x^* e|=v^*(e+|x^* e|)^{-1}|x^* e|\in \mathcal{M}_e\subseteq \mathcal{M}$.  Since $u_1$ is a partial isometry in $\mathcal{M}$, $h_1 ex=h_1|x^* e|u_1\in\mathcal{M}$.  Therefore, (ii) holds.

        The proof for $h_3$ and $h_4$ is similar.
    \end{proof}

The following Lemma is also helpful.

    \begin{lemma}
        Suppose $\mathcal{M}$ is a von Neumann algebra with a semifinite, faithful, normal tracial weight $\tau$.  Let $H^\infty$ be a semifinite, subdiagonal subalgebra with respect to $(\mathcal{M},\Phi)$, where $\Phi$ is a faithful, normal conditional expectation from $\mathcal{M}$ onto $\mathcal{D}=H^\infty\cap (H^\infty)^*$.

        There exists a net $\{e_\lambda\}_{\lambda\in\Lambda}$ of projections in $\mathcal{D}$ such that
        \begin{enumerate}
            \item[(i)] $e_\lambda \rightarrow I$ in the weak* topology on $\mathcal{M}$, and $\tau(e_\lambda)<\infty$ for all $\lambda\in\Lambda$.
            \item[(ii)] For every $x\in L^\alpha(\mathcal{M},\tau)$,
            \begin{equation}
                \lim_\lambda \alpha(e_\lambda x-x)=0; \quad \lim_\lambda \alpha(xe_\lambda -x)=0; \,\text{ and }\, \lim_\lambda \alpha(e_\lambda x e_\lambda-x)=0. \notag
            \end{equation}
        \end{enumerate}
    \end{lemma}

    \begin{proof}
        We know that $H^\infty$ is a semifinite subdiagonal subalgebra of $\mathcal{M}$, therefore the restriction of $\tau$ to $\mathcal{D}$ is semifinite.  From Lemma 2.2 in \cite{Sag}, there exists a net of projections $\{e_\lambda\}_{\lambda\in\Lambda}$ in $\mathcal{D}$ such that $e_\lambda\rightarrow I$ in the weak* topology on $\mathcal{D}$, and $\tau(e_\lambda)<\infty$ for all $\lambda\in\Lambda$. Therefore,
            $$\lim_\lambda |\tau(e_\lambda z-z)|=0 \text{ for every }z\in L^1(\mathcal{D},\tau).$$
        Also, for each $y$ in $L^1(\mathcal{M},\tau)$, we have that
            $$\lim_\lambda |\tau(e_\lambda y-y)|=\lim_\lambda |\tau(\Phi(e_\lambda y-y))|=\lim_\lambda |\tau(e_\lambda \Phi(y)-\Phi(y))|=0.$$
        Namely, $e_\lambda \rightarrow I$ in the weak* topology on $\mathcal{M}$, and $\tau(e_\lambda)<\infty$ for every $\lambda\in\Lambda$.  (i) is satisfied.

        Then from (i) and Definition \ref{def3.1}, we may conclude that (ii) holds. Namely, for every $x\in L^\alpha(\mathcal{M},\tau)$,
            $$\lim_\lambda \alpha(e_\lambda x-x)=0;\quad \lim_\lambda \alpha(xe_\lambda -x)=0; \,\text{ and }\, \lim_\lambda \alpha(e_\lambda x e_\lambda -x)=0.$$

        Therefore, the Lemma is proven.
    \end{proof}
Finally, we recall the definition of a row sum of subspaces of $L^\alpha(\mathcal{M},\tau)$.

    \begin{definition}\label{def4.7}
        Let $\mathcal{M}$ be a von Neumann algebra with a semifinite, normal, faithful tracial weight $\tau$.  Suppose $X$ is a closed subspace of $L^\alpha(\mathcal{M},\tau)$, and $\{X_i\}_{i\in\mathcal{I}}$ are closed subspaces of $L^\alpha(\mathcal{M},\tau)$.  If
        \begin{enumerate}
            \item $X_j X_i^*=\{0\}$ for every $i,j\in\mathcal{I}$, $i\neq j$; and
            \item $X=[span\{X_i:i\in\mathcal{I}\}]_\alpha$,
        \end{enumerate}
        we call $X$ the internal row sum of $\{X_i\}_{i\in\mathcal{I}}$, and denote it by $X=\oplus^{row}_{i\in\mathcal{I}} X_i$.   Also, we denote $span\{X_i:i\in\mathcal{I}\}$ by $\sum_{i\in\mathcal{I}} X_i$.
    \end{definition}

\section{Beurling Theorem for Semifinite Hardy Spaces with Norm $\alpha$}

    \begin{theorem}\label{theorem3.1}
        Let $\mathcal{M}$ be a von Neumann algebra with a faithful, normal semifinite tracial weight $\tau$, and $H^\infty$ be a semifinite subdiagonal subalgebra of $\mathcal{M}$.  Let $\alpha$ be a  unitarily invariant, locally $\|\cdot\|_1$-dominating, mutually continuous norm with respect to $\tau$.  Let $\mathcal{D}=H^\infty\cap (H^\infty)^*$.  Assume that $\mathcal{K}$ is a closed subspace of $L^\alpha(\mathcal{M},\tau)$ such that $H^\infty \mathcal{K}\subseteq \mathcal{K}$.  

        Then, there exist a closed subspace $Y$ of $L^\alpha(\mathcal{M},\tau)$ and a family $\{u_\lambda\}$ of partial isometries in $\mathcal{M}$ such that
            \begin{enumerate}
                \item[(i)] $u_\lambda Y^*=0$ for every $\lambda\in\Lambda$;\
                \item[(ii)] $u_\lambda u_\lambda^*\in\mathcal{D}$, and $u_\lambda u_\mu^*=0$ for every $\lambda, \mu\in \Lambda$ with $\lambda\neq\mu$;
                \item[(iii)] $Y=[H^\infty_0 Y]_\alpha$;
                \item[(iv)] $\mathcal{K}=Y\oplus^{row}(\oplus^{row}_{\lambda\in\Lambda} H^\alpha u_\lambda)$.
            \end{enumerate}
    \end{theorem}

        First, we prove some lemmas.

    \begin{lemma}\label{lemma3.2}
        Suppose $\mathcal{M}$ is a von Neumann algebra with a faithful, normal, semifinite tracial weight $\tau$, and that $H^\infty$ is a semifinite, subdiagonal subalgebra of $\mathcal{M}$.  Suppose also that $\alpha$ is a unitarily invariant, locally $\|\cdot\|_1$-dominating, mutually continuous norm with respect to $\tau$.  Assume that $\mathcal{K}$ is a closed subspace of $L^\alpha (\mathcal{M},\tau)$ such that $H^\infty \mathcal{K}\subseteq\mathcal{K}$. Then the following hold:
        \begin{enumerate}
            \item[(i)] $\mathcal{K}\cap \mathcal{M} = \overline{\mathcal{K}\cap\mathcal{M}}^{w^*} \cap L^\alpha(\mathcal{M},\tau)$
            \item[(ii)] $\mathcal{K}=[\mathcal{K}\cap \mathcal{M}]_\alpha$
        \end{enumerate}
    \end{lemma}
    \begin{proof}
    {}  (i)
        It is clear that
            $$\mathcal{K}\cap\mathcal{M}\subseteq \overline{\mathcal{K}\cap\mathcal{M}}^{w^*}\cap L^\alpha(\mathcal{M},\tau).$$

        We will prove that
            $$\mathcal{K}\cap \mathcal{M}=\overline{\mathcal{K}\cap \mathcal{M}}^{w^*}\cap L^\alpha(\mathcal{M},\tau).$$
        Assume, to the contrary, that $\mathcal{K}\cap\mathcal{M}\subsetneqq \overline{\mathcal{K}\cap\mathcal{M}}^{w^*}\cap L^\alpha(\mathcal{M},\tau)$.
        Then there exists an $x\in  \overline{\mathcal{K}\cap\mathcal{M}}^{w^*}\cap L^\alpha(\mathcal{M},\tau)$, with $x\notin \mathcal{K}\cap\mathcal{M}$. By the Hahn-Banach theorem, there exists a $\varphi\in L^\alpha(\mathcal{M},\tau)^\#$ such that $\varphi(x)\neq 0$, and $\varphi(y)=0$ for every $y\in \mathcal{K}\cap\mathcal{M}$.

        Since the restriction of $\tau$ to $\mathcal{D}=H^\infty\cap (H^\infty)^*$ is semifinite, there exists a family $\{e_\lambda\}$ of projections in $\mathcal{D}$ such that $\tau(e_\lambda)<\infty$ for every $\lambda$, and $e_\lambda\rightarrow I$ in the weak* topology.  This implies that $e_\lambda x\rightarrow x$ in the weak* topology and in $\alpha$-norm by condition (3a) of definition \ref{def3.1}.

        Thus, there must exist a $\lambda$ such that $e_\lambda x\notin \mathcal{K}\cap \mathcal{M}$.  Also, $e_\lambda x \in e_\lambda L^\alpha(\mathcal{M},\tau)$.

        Define $\psi: \mathcal{M}\rightarrow \mathbb{C}$ by $\psi(z)=\varphi(e_\lambda z)$ for every $z\in\mathcal{M}$.  Then $\psi$ is a bounded linear functional.  We will show that $\psi$ is normal, i.e. for an increasing net ${f_\mu}$ of projections in $\mathcal{M}$ such that $f_\mu \rightarrow I$ in weak$^*$-topology, then $\psi(f_\mu)\rightarrow \psi(I)$.  By condition (3a) of Defintion \ref{def3.1}, we get that $\alpha(e_\lambda f_\mu - e_\lambda I)\rightarrow 0$, for a fixed $\lambda$.
        Since $\varphi\in L^\alpha (\mathcal{M},\tau)^\#$, $\varphi(e_\lambda f_\mu)\rightarrow \varphi(e_\lambda I)$.  However
            $$\varphi(e_\lambda f_\mu)=\psi(f_\mu),$$
        and $\varphi(e_\lambda I)=\psi(I)$.  Thus, $\psi(f_\mu)\rightarrow \psi(I)$.  Therefore, $\psi$ is a normal, bounded linear functional, namely, $\psi\in L^{1}(\mathcal{M},\tau)$.

        There exists a $\xi\in L^1(\mathcal{M},\tau)$ such that $\psi(z)=\tau(z\xi)$ for every $z\in\mathcal{M}$.  Note that $\psi(x)=\varphi(e_\lambda x)=\tau(x\xi)\neq 0$.  Thus, there exists a projection $e\in\mathcal{D}$ such that $\tau(e)<\infty$ so that $\psi(ex)=\varphi(e_\lambda ex)=\tau(ex\xi)\neq 0$, and $\psi(ey)=\varphi(e_\lambda ey)=\tau(ey\xi)=0$ for every $y\in \mathcal{K}\cap\mathcal{M}$.

        Recall that $x\in\overline{\mathcal{K}\cap\mathcal{M}}^{w^*}$.  Therefore, there exists a sequence $\{y_\mu\}$ in $\mathcal{K}\cap\mathcal{M}$ such that $y_\mu \rightarrow x$ in the weak* topology. Note that  $\xi e \in L^1(\mathcal{M},\tau)$.  Hence,
            $$\tau(y_\mu \xi e)\rightarrow \tau(x\xi e).$$
        However, $\tau(y_\mu \xi e)=0$, so $\tau(x \xi e)=0$, which is a contradiction. Therefore (i) is proven.

        (ii) Clearly, $\mathcal{K}\cap \mathcal{M}\subseteq \mathcal{K}$, and $\mathcal{K}$ is $\alpha$-norm closed, so
            $$[\mathcal{K}\cap\mathcal{M}]_\alpha \subseteq \mathcal{K}.$$

        We will show that
            $$\mathcal{K}=[\mathcal{K}\cap\mathcal{M}]_\alpha.$$
        Suppose to the contrary, that $[\mathcal{K}\cap \mathcal{M}]_\alpha\subsetneqq\mathcal{K}$.  There exists an $x\in \mathcal{K}$ such that $x\notin [\mathcal{K}\cap\mathcal{M}]_\alpha$.  We know that $\mathcal{D}$ is semifinite, so there exists a family of projections $\{e_\lambda\}_{\lambda\in\Lambda}$ such that $\tau(e_\lambda)<\infty$, and $e_\lambda \rightarrow I$ in the weak-* topology.  By Definition \ref{def3.1}, part (3a), $e_\lambda x\rightarrow x$ in $\alpha$-norm.  So, there exists $\lambda$ such that $e_\lambda x\in \mathcal{K}$, since $x\in \mathcal{K}$, and $e_\lambda x\notin [\mathcal{K}\cap \mathcal{M}]_\alpha$, as $x\notin [\mathcal{K}\cap\mathcal{M}]_\alpha$.

        By Lemma \ref{lemma2.12}, there exist an $h_1\in e_\lambda H^\infty e_\lambda$ and an $h_2\in e_\lambda H^\alpha e_\lambda$ such that $h_1 e_\lambda x \in \mathcal{M}$, and $h_1 h_2 =e_\lambda = h_2 h_1$.  Thus, $e_\lambda x = h_2 h_1 e_\lambda x$, $h_1 e_\lambda x\in\mathcal{M}$, and $h_1 e_\lambda x\in \mathcal{K}$, since $H^\infty \mathcal{K}\subseteq \mathcal{K}$.  Also, $h_2\in e_\lambda H^\alpha e_\lambda$, so there exists a sequence $\{a_n\}$ in $H^\infty$ such that $a_n\rightarrow h_2$ in $\alpha$-norm.  Hence, $e_\lambda x=h_2 h_1 ex$,  $a_n h_1 e_\lambda x \in K\cap \mathcal{M}$, and
            $$a_n h_1 e_\lambda x\rightarrow h_2 h_1 ex$$
        in $\alpha$-norm.

        Therefore, $e_\lambda x\in[\mathcal{K}\cap\mathcal{M}]_\alpha$, which is a contradiction.  Thus, (ii) is proven.
    \end{proof}

    \begin{lemma}\label{lemma3.3}
        Suppose $\mathcal{M}$ is a  von Neumann algebra with a faithful, normal, semifinite tracial weight $\tau$, and suppose that $\alpha$ is a unitarily invariant, locally $\|\cdot\|_1$-dominating, mutually continuous norm with respect to $\tau$.  Let $H^\infty$ be a semifinite, subdiagonal subalgebra of $\mathcal{M}$.  Assume that $\mathcal{K}$ is a weak* closed subspace of $\mathcal{M}$ such that $H^\infty \mathcal{K}\subseteq \mathcal{K}$. Then
            $$\mathcal{K}=\overline{[\mathcal{K}\cap L^\alpha(\mathcal{M},\tau)]_\alpha \cap \mathcal{M}}^{w^*}$$.
    \end{lemma}

    \begin{proof}
        First we must show that
            $$\mathcal{K}\subseteq \overline{[\mathcal{K}\cap L^\alpha(\mathcal{M},\tau)]_\alpha\cap\mathcal{M}}^{w^*}.$$

        Let $x\in\mathcal{K}\subseteq \mathcal{M}$. We know that $\tau$ restricted to $\mathcal{D}$ is semifinite, so there exists a net of projections $\{e_\lambda\}_{\lambda\in\Lambda}$ such that $\tau(e_\lambda)<\infty$ and $e_\lambda \rightarrow I$ in the weak* topology. Also, $e_\lambda x\rightarrow x$ in the weak* topology.

        To show that
            $$x\in\overline{[\mathcal{K}\cap L^\alpha(\mathcal{M},\tau)]_\alpha \cap \mathcal{M}}^{w^*},$$
        it is sufficient to show that $e_\lambda x\in [\mathcal{K}\cap L^\alpha(\mathcal{M},\tau)]_\alpha\cap \mathcal{M}$.  We have that $e_\lambda x$ is in $\mathcal{K}$, as $x\in\mathcal{K}$ and  $\mathcal{K}$ is $H^\infty$-invariant.  We also know $\|e_\lambda x\|_\alpha\leq \|e_\lambda\|_\alpha\|x\|<\infty$.  Therefore, $e_\lambda x\in L^\alpha(\mathcal{M},\tau)$, and $e_\lambda x\in\mathcal{K}\cap L^\alpha(\mathcal{M},\tau)\subseteq [\mathcal{K}\cap L^\alpha(\mathcal{M},\tau)]_\alpha$.  Thus, $x\in\overline{[\mathcal{K}\cap L^\alpha(\mathcal{M},\tau)]_\alpha\cap \mathcal{M}}^{w^*}$.

        Hence $\mathcal{K}\subseteq\overline{[\mathcal{K}\cap L^\alpha(\mathcal{M},\tau)]_\alpha \cap\mathcal{M}}^{w^*}$.

        Next, we show that
            $$\overline{[\mathcal{K}\cap L^\alpha(\mathcal{M},\tau)]_\alpha \cap\mathcal{M}}^{w^*}\subseteq \mathcal{K}.$$
        It suffices to show that $[\mathcal{K}\cap L^\alpha(\mathcal{M},\tau)]_\alpha\cap\mathcal{M} \subseteq \mathcal{K}$ since $\mathcal{K}$ is weak*-closed.

        Suppose, to the contrary, that $ [\mathcal{K}\cap L^\alpha (\mathcal{M},\tau)]_\alpha\cap \mathcal{M}  \subsetneqq \mathcal{K} $.  There exists an $x\in[\mathcal{K}\cap L^\alpha(\mathcal{M},\tau)]_\alpha \cap \mathcal{M}$ such that $x\notin \mathcal{K}$.
        Since the restriction of $\tau$ to $\mathcal{D}$ is semifinite, there exists a net $\{e_\lambda\}_{\lambda\in\Lambda}$ of projections such that $\tau(e_\lambda)\leq\infty$ and $e_\lambda x\rightarrow x$ in the weak* topology.

    As $x\notin \mathcal K$, by the Hahn-Banach theorem, there exists a $\varphi\in\mathcal{M}_\#$ such that $\varphi(x)\neq 0$ and $\varphi(y)=0$ for all $y$ in $\mathcal{K}$.  As $x\in[\mathcal{K}\cap L^\alpha (\mathcal{M},\tau)]_\alpha \cap\mathcal{M}$ and $x\notin \mathcal{K}$, there exists a $\lambda$ such that $e_\lambda x\in [\mathcal{K}\cap L^\alpha (\mathcal{M},\tau)]_\alpha \cap \mathcal{M}$ and $e_\lambda x\notin \mathcal{K}$.
        Since $\varphi\in\mathcal{M}_\#$, there exists a $\xi$ in $L^1(\mathcal{M},\tau)$ such that $\varphi(z)=\tau(z\xi)$ for every $z\in\mathcal{M}$.  It follows that there exists a projection $e\in\mathcal{D}$ with $\tau(e)<\infty$  so that $\tau(x\xi e)\neq 0$, and $\tau(y\xi e)=0$ for every $y\in\mathcal{K}$.

        We claim that there exists a $z=\xi e\in\mathcal{M}e$ such that $\tau(xz)\neq 0$ and $\tau(yz)=0$ for all $y\in\mathcal{K}$.

        Note that $\xi e\in L^1(\mathcal{M},\tau)$ since $\xi \in L^1(\mathcal{M},\tau)$ and $\tau(e)<\infty$. By Lemma \ref{lemma2.12}, there exist $h_3\in eH^\infty e$, and $h_4\in eH^1 e$ such that $h_3 h_4=e=h_4 h_3$ and $\xi eh_3\in \mathcal{M}$.  There exists $\{k_n\}$ in $H^\infty$ such that $k_n\rightarrow h_4$ in $\|\cdot\|_1$-norm.  So,
            \begin{align}
                \lim_{n\rightarrow \infty} |\tau(ex\xi)-\tau(x\xi eh_3k_n)| &=\lim_{n\rightarrow \infty}|\tau(x\xi eh_3h_4)-\tau(x\xi eh_3k_n)|\notag \\
            &\leq \lim_{n\rightarrow\infty}\|x\| \|\xi e h_3\| \|h_4-k_n\|_1\notag \\
            &=0. \notag
            \end{align}
        There exists an $N\in\mathbb{N}$ such that $\tau(x\xi e h_3 k_N)\neq 0$, since $\tau(x\xi)\neq 0$.  We let $z=\xi e h_3 k_N \in \mathcal{M}$.  Then, $z=ze\in\mathcal{M}e$ such that $\tau(xz)=\tau(x \xi e h_3 k_N)\neq 0$, and $\tau(yz)=\tau(y\xi e h_3 k_N)=\tau((e h_3 k_N)y\xi)=0$ for every $y\in\mathcal{K}$.

        Since $x\in[\mathcal{K}\cap L^\alpha(\mathcal{M},\tau)]_\alpha \cap \mathcal{M}$ there exists $\{x_n\}$ in $\mathcal{K}\cap L^\alpha(\mathcal{M},\tau)$ such that $x_n\rightarrow x$ in $\alpha$ norm, and $ex_n\rightarrow ex$ in $\alpha$-norm. Note $ey=\sqrt{eyy^*e}v= e\sqrt{eyy^*e}ev$.  Therefore, $ex_n\rightarrow ex$ in $\|\cdot\|_1$-norm, as  $\|ey\|_1=\|e\sqrt{eyy^*e}e\|_1$,   $\alpha(ey)=\alpha(e\sqrt{eyy^*e}e)$, and  $\alpha$ is locally $\|\cdot\|_1$ -dominating.

         We also have that $|\tau(xz-x_n z)|=|\tau((x-x_n)z)|\leq \|e(x_n-x)\|_1 \|z\|$.  Finally, since $\{x_n\}$ is in $\mathcal{K}\cap L^\alpha (\mathcal{M}, \tau) \subseteq \mathcal{K}$, $\tau(x_n z)=0$.  Hence, $\tau(xz)=0$, which is a contradiction.

         Therefore, $\overline{[\mathcal{K}\cap L^\alpha(\mathcal{M},\tau)]_\alpha\cap\mathcal{M}}^{w^*}\subseteq\mathcal{K}$.

         Thus, $\mathcal{K}=\overline{[\mathcal{K}\cap L^\alpha(\mathcal{M},\tau)]_\alpha \cap \mathcal{M}}^{w^*}$.
    \end{proof}

    \begin{lemma}\label{lemma3.4}
        Suppose $\mathcal{M}$ is a semifinite von Neumann algebra with a faithful, normal tracial weight $\tau$, and suppose that $\alpha$ is a unitarily invariant, locally $\|\cdot\|_1$-dominating, mutually-continuous norm with respect to $\tau$.  Let $H^\infty$ be a semifinite, subdiagonal subalgebra of $\mathcal{M}$.  Assume that $S$ is a subset of $\mathcal{M}$ such that $H^\infty S\subseteq S$.  Then
            $$[S\cap L^\alpha(\mathcal{M},\tau)]_\alpha= [\overline{S}^{w^*}\cap L^\alpha (\mathcal{M}, \tau)]_\alpha.$$
    \end{lemma}

    \begin{proof}
        Clearly, $S\cap L^\alpha(\mathcal{M},\tau)\subseteq \overline{S}^{w^*} \cap L^\alpha(\mathcal{M},\tau)$ so, $[S\cap L^\alpha(\mathcal{M},\tau)]_\alpha \subseteq [\overline{S}^{w^*} \cap L^\alpha(\mathcal{M},\tau)]_\alpha$.\\

        We will show that $\overline{S}^{w^*}\cap L^\alpha(\mathcal{M},\tau)\subseteq [S\cap L^\alpha(\mathcal{M},\tau)]_\alpha$.  Let $x\in\overline{S}^{w^*}\cap L^\alpha(\mathcal{M},\tau)$.  We know that there exists a net $\{e_\lambda\}$ in $\mathcal{D}$ of projections such that $\tau(e_\lambda)<\infty$, and $e_\lambda \rightarrow I$ in the weak* topology.  Thus, $e_\lambda x\rightarrow x$ in the weak* topology.

 We will show that $e_\lambda x \in [S\cap L^\alpha(\mathcal{M},\tau)]_\alpha$ in order to show that $x\in[S\cap L^\alpha(\mathcal{M},\tau]_\alpha$.  By Lemma \ref{lemma3.2}, we have that
        $$[S\cap L^\alpha(\mathcal{M},\tau)]_\alpha\cap\mathcal{M}\subseteq \overline{[S\cap L^\alpha(\mathcal{M},\tau)]_\alpha}^{w^*}\cap L^\alpha(\mathcal{M},\tau).$$
        Since $x\in\overline{S}^{w^*}\cap L^\alpha(\mathcal{M},\tau)$, there exists a net $\{x_j\}$ in $S$ such that $x_j \rightarrow x$ in the weak*-topology.
        Therefore  $e_\lambda x_j \rightarrow e_\lambda x$ in the weak*-topology for every $\lambda\in\Lambda$. We note that $\alpha(e_\lambda x_j)\leq \alpha(e_\lambda)\|x_j\|$, and $H^\infty S\subseteq S$.
        Therefore $e_\lambda x_j\in S\cap L^\alpha(\mathcal{M},\tau)$, and $e_\lambda x_j\in \overline{[S\cap L^\alpha(\mathcal{M},\tau)]_\alpha\cap\mathcal{M}}^{w^*}$.
        Thus, $e_\lambda x\in\overline{[S\cap L^\alpha[\mathcal{M},\tau)]_\alpha \cap\mathcal{M}}^{w^*}$.
    It is clear that $e_\lambda x\in L^\alpha(\mathcal{M},\tau)$.   By Lemma \ref{lemma3.2}, $\overline{[S\cap L^\alpha(\mathcal{M},\tau)]_\alpha \cap \mathcal{M}}^{w^*} \cap L^\alpha(\mathcal{M},\tau) = [S\cap L^\alpha(\mathcal{M},\tau)]_\alpha\cap\mathcal{M}$.  So $e_\lambda x\in[S\cap L^\alpha(\mathcal{M},\tau)]_\alpha$.

Therefore,  $ x\in[S\cap L^\alpha(\mathcal{M},\tau)]_\alpha$, whence $\overline{S}^{w^*}\cap L^\alpha(\mathcal{M},\tau\subseteq [S\cap L^\alpha(\mathcal{M},\tau)]_\alpha$.  Hence,
            $$[\overline{S}^{w^*}\cap L^\alpha(\mathcal{M},\tau)]_\alpha=[S\cap L^\alpha(\mathcal{M},\tau)]_\alpha.$$

    \end{proof}

    Now, we prove Theorem \ref{theorem3.1}.

    \begin{proof}
        Let $\mathcal{K}_1=\overline{\mathcal{K}\cap\mathcal{M}}^{w^*}$.  $\mathcal{K}_1$ is a weak* closed subspace of $\mathcal{M}$ such that $H^\infty \mathcal{K}_1\subseteq \mathcal{K}_1$.  Then by Theorem 4.5 in \cite{Sag}, there exist a weak* closed subspace $Y_1\subseteq\mathcal{M}$ and a family $\{u_\lambda\}_{\lambda\in\Lambda}$ of partial isometries in $\mathcal{M}$ such that
            \begin{enumerate}
                \item[(a)] $u_\lambda Y_1^*=0$ for every $\lambda\in\Lambda$;
                \item[(b)] $u_\lambda u_\lambda^* \in\mathcal{D}$, and $u_\lambda u_\mu^*=0$ for every $\lambda, \mu \in\Lambda$ such that $\lambda\neq\mu$;
                \item[(c)] $Y_1=\overline{H^\infty_0 Y_1}^{w^*}$;
                \item[(d)] $\mathcal{K}_1=Y_1\oplus^{row}(\oplus^{row}_{\lambda\in\Lambda} H^\infty u_\lambda)$.
            \end{enumerate}
        Let $Y=[Y_1\cap L^\alpha(\mathcal{M},\tau)]_\alpha$. \\
        (i) We know that there exists $\{a_n\}\subseteq Y_1^*$ such that $a_n\rightarrow a$ in $\alpha$-norm for some $a\in Y^*_1$. From (a), and the definition of $Y_1$, $a_n u_i \rightarrow au_i$ in $\alpha$-norm.  Thus, we may conclude that $u_\lambda Y^*=0$ for every $\lambda\in\Lambda$.

        (ii) follows directly from (b).  

        (iii) We will show that $Y=[H^\infty_0 Y]_\alpha$.  We have that
                \begin{align}
                    Y&=[Y_1\cap L^\alpha(\mathcal{M},\tau)]_\alpha \tag{by definition of $Y$}  \nonumber \\
                    &=[\overline{H^\infty_0 Y_1}^{w^*}\cap L^\alpha(\mathcal{M},\tau)]_\alpha \tag{by (c)} \nonumber \\
                    &=[H^\infty_0 Y_1\cap L^\alpha(\mathcal{M},\tau)]_\alpha \tag{by Lemma \ref{lemma3.4}} \nonumber\\
                    &=[H^\infty_0(\overline{[Y_1\cap L^\alpha(\mathcal{M},\tau)]_\alpha \cap \mathcal{M}}^{w^*})\cap L^\alpha(\mathcal{M},\tau)]_\alpha  \tag{by Lemma \ref{lemma3.3}}\nonumber \\
                    &\subseteq [\overline{H^\infty_0 ([Y_1\cap L^\alpha(\mathcal{M},\tau)]_\alpha \cap\mathcal{M})}^{w^*}\cap L^\alpha(\mathcal{M},\tau)]_\alpha \tag{by Theorem 1.7.8 in \cite{Sakai}}\nonumber \\
                    &=[H^\infty_0([Y_1\cap L^p(\mathcal{M},\tau)]_\alpha \cap \mathcal{M}) \cap L^\alpha(\mathcal{M},\tau)]_\alpha \tag{by Lemma \ref{lemma3.4}} \nonumber\\
                    &=[H^\infty_0(Y\cap \mathcal{M})\cap L^\alpha(\mathcal{M},\tau)]_\alpha \tag{by defintion of $Y$} \nonumber \\
                    &\subseteq[H^\infty_0 Y]_\alpha \nonumber\\
                    &\subseteq Y. \nonumber
                \end{align}
                Hence, $Y=[H^\infty_0 Y]_\alpha$ as desired.

        (iv) Finally, we will show that $\mathcal{K}=Y\oplus^{row}(\oplus^{row}_{\lambda\in\Lambda} H^\alpha u_\lambda$). \\
        Recall that $Y=[Y_1\cap L^\alpha(\mathcal{M},\tau)]_\alpha$.

        We claim that $[H^\infty_0Y_1\cap L^\alpha(\mathcal{M},\tau)]_\alpha\subseteq [H^\infty_0(Y_1\cap L^\alpha(\mathcal{M},\tau)]_\alpha$.

        Also, by Lemma \ref{lemma3.2}, $H^\alpha u_\lambda =[H^\infty u_\lambda\cap L^\alpha(\mathcal{M},\tau)]_\alpha$ for every $\lambda\in\Lambda$.  Now,
            \begin{align}
                \mathcal{K}&=[\mathcal{K}_1\cap L^\alpha(\mathcal{M},\tau)]_\alpha \nonumber\\
                &=[\overline{Y_1+\sum_{\lambda\in\Lambda} H^\infty u_\lambda}^{w^*}\cap L^\alpha(\mathcal{M},\tau)]_\alpha  \tag{by definition of $\mathcal{K}_1$} \nonumber\\
                &=[ Y_1+\sum_{\lambda\in\Lambda} H^\infty u_\lambda \cap L^\alpha(\mathcal{M},\tau)]_\alpha \tag{by Lemma \ref{lemma3.4}} \nonumber \\
                &=[Y_1\cap L^\alpha(\mathcal{M},\tau)+\sum_{\lambda\in\Lambda} H^\infty u_\lambda \cap L^\alpha(\mathcal{M},\tau)]_\alpha \tag{by (a) and (b)} \nonumber\\
                &=[Y+\sum_{\lambda\in\Lambda} H^\alpha u_\lambda]_\alpha\nonumber \\
                &=Y\oplus^{row}(\oplus^{row}_{\lambda\in\Lambda}H^\alpha u_\lambda)\nonumber
            \end{align}
            where the last equality comes from Definition \ref{def4.7}.
    \end{proof}

    \begin{corollary} \label{corollary5.5}
        Suppose that $\mathcal{M}$ is a von Neumann algebra with a faithful, normal, semifinite tracial weight $\tau$.  Let $\alpha$ be a unitarily invariant, locally $\|\cdot\|_1$-dominating, mutually continuous norm with respect to $\tau$.  Let $\mathcal{K}$ be a subset of $L^\alpha$ such that $\mathcal{M}\mathcal{K}\subseteq\mathcal{K}$.  Then there exists a projection $q$ with $\mathcal{K}=\mathcal{M}q$.
    \end{corollary}

    \begin{proof}
        We note that $\mathcal{M}$ can be considered as a semifinite subdiagonal subalgebra of $\mathcal{M}$ itself.  Hence, we let $\mathcal{M}=H^\infty$, and it follows that $\mathcal{D}=\mathcal{M}$ and $\Phi$ is the identity map on $\mathcal{M}$.  Also, $H^\infty_0=\{0\}$ and $H^\alpha=L^\alpha(\mathcal{M},\tau)$.

        Let $\mathcal{K}$ be a closed subspace of $L^\alpha(\mathcal{M},\tau)$ such that $\mathcal{MK}\subseteq\mathcal{K}$.  From Theorem \ref{theorem3.1},
            $$\mathcal{K}=Y\oplus_{row}(\oplus^{row}_{\lambda\in\Lambda}H^\alpha u_\lambda),$$
        where $u_\lambda Y^*=0$ for every $\lambda\in\Lambda$, $u_\lambda u_\lambda^*\in\mathcal{D}$, and $u_\lambda u_mu^*=0$ for every $\lambda,\mu\in\Lambda$ such that $\lambda\neq\mu$, and $Y=[H^\infty_0 Y]_\alpha$.

        It is clear that because $H^\infty_0=\{0\}$, $Y=0$.  Also, since $\mathcal{D}=\mathcal{M}$, we have that
            \begin{align}
                H^\alpha u_\lambda &= L^\alpha(\mathcal{M},\tau) u_\lambda =L^\alpha(\mathcal{M},\tau)u_\lambda u_\lambda^* u_\lambda\notag \\
                &\subseteq L^\alpha(\mathcal{M},\tau)u_\lambda^* u_\lambda \subseteq L^\alpha(\mathcal{M},\tau)u_\lambda =H^\alpha u_\lambda. \notag
            \end{align}
            Therefore, $H^\alpha u_\lambda =L^\alpha(\mathcal{M},\tau u_\lambda^* u_\lambda$.  Specifically, we find that
            \begin{align}
                \mathcal{K} &= Y\oplus^{row}(\oplus^{row}_{\lambda\in\Lambda} H^\alpha u_\lambda)=(\oplus^{row}_{\lambda\in\Lambda}L^\alpha(\mathcal{M},\tau)u_\lambda^* u_\lambda)\notag\\
                &L^\alpha(\mathcal{M},\tau)\big(\sum_{\lambda\in\Lambda}u_\lambda^* u_\lambda\big)=L^\alpha(\mathcal{M},\tau)q \notag
            \end{align}
            where we let $\sum_{\lambda\in\Lambda} u_\lambda^* u_\lambda =q$, and $q$ is a projection in $\mathcal{M}$.  This ends the proof.
    \end{proof}

\section{Applications}

    \subsection{Invariant subspaces for non-commutative Banach function spaces}
We briefly recall our discussion of a non-commutative Banach function space.  Let $E$ be a symmetric Banach function space on $(0,\infty)$ with Lebesgue measure.
As before, we let $\mathcal{M}$ be a von Neumann algebra with a faithful, normal tracial state $\tau$ and $\mathcal{I}=\{x\in\mathcal{M}:x \text{ is a finite rank operator in }(\mathcal{M},\tau) \text{ and }\\ \|\mu(x)\|_E<\infty\}$.  We may then define a Banach function space $\mathcal{I}(\tau)$ , and a norm $\|\cdot\|_{\mathcal{I}(\tau)}$ by $\|x\|_{\mathcal{I}(\tau)}=\|\mu(x)\|_{E(0,\infty)}$ for every $x\in \mathcal{I}(\tau)$.
We let $H^\infty$ be a semifinite subdiagonal subalgebra of $\mathcal M$, as described earlier.
The following is an easy corollary of Theorem \ref{theorem3.1} and Proposition \ref{prop3.3}.

    \begin{corollary}\label{corollary6.1}
        Suppose that $\mathcal{I}(\tau)$ is a Banach function space on the diffuse von Neumann algebra $\mathcal{M}$ with order continuous norm $\|\cdot\|_{\mathcal{I}(\tau)}$. Let $\mathcal{D}=H^\infty\cap (H^\infty)^*$.  Assume that $\mathcal{K}$ is a closed subspace of $\mathcal{I}(\tau)$ such that $H^\infty \mathcal{K}\subseteq \mathcal{K}$.  

        Then, there exist a closed subspace $Y$ of $\mathcal{I}(\tau)$ and a family $\{u_\lambda\}$ of partial isometries in $\mathcal{M}$ such that
            \begin{enumerate}
                \item[(i)] $u_\lambda Y^*=0$ for every $\lambda\in\Lambda$;\
                \item[(ii)] $u_\lambda u_\lambda^*\in\mathcal{D}$, and $u_\lambda u_\mu^*$ for every $\lambda, \mu\in \Lambda$ with $\lambda\neq\mu$;
                \item[(iii)] $Y=[H^\infty_0 Y]_\alpha$;
                \item[(iv)] $\mathcal{K}=Y\oplus^{row}(\oplus^{row}_{\lambda\in\Lambda} H^{\mathcal{I}(\tau)} u_\lambda)$.
            \end{enumerate}
    \end{corollary}

    \subsection{Invariant subspaces for factors}
        We also have the following corollary from Theorem \ref{theorem3.1} and Proposition \ref{prop3.5}.  

    \begin{corollary} \label{corollary6.2}
        Suppose $\mathcal{M}$ is a factor with a faithful, normal tracial weight $\tau$.  Let $\alpha:\mathcal{I}\rightarrow[0,\infty)$, where $\mathcal{I}$ is the set of elementary operators in $\mathcal{M}$, be a unitarily invariant norm such that any net $\{e_\lambda\}$ in $\mathcal{M}$ with $e_\lambda \uparrow I$ in the weak* topology implies that $\alpha((e_\lambda - I)x)\rightarrow 0$.  Let $H^\infty$ be a semifinite subdiagonal subalgebra of $L^\alpha(\mathcal{M},\tau)$.  Let $\mathcal{D}=H^\infty\cap (H^\infty)^*$.  Assume that $\mathcal{K}$ is a closed subspace of $L^\alpha(\mathcal{M},\tau)$ such that $H^\infty \mathcal{K}\subseteq \mathcal{K}$.  

        Then, there exist a closed subspace $Y$ of $L^\alpha(\mathcal{M},\tau)$ and a family $\{u_\lambda\}$ of partial isometries in $\mathcal{M}$ such that
            \begin{enumerate}
                \item[(i)] $u_\lambda Y^*=0$ for every $\lambda\in\Lambda$;\
                \item[(ii)] $u_\lambda u_\lambda^*\in\mathcal{D}$, and $u_\lambda u_\mu^*$ for every $\lambda, \mu\in \Lambda$ with $\lambda\neq\mu$;
                \item[(iii)] $Y=[H^\infty_0 Y]_\alpha$;
                \item[(iv)] $\mathcal{K}=Y\oplus^{row}(\oplus^{row}_{\lambda\in\Lambda} H^\alpha u_\lambda)$.
            \end{enumerate}
    \end{corollary}

    \subsection{Invariant subspaces of analytic crossed products}
        Suppose that $\mathcal{M}$ is a von Neumann algebra with a semifinite, faithful normal tracial state $\tau$.  We let $\beta$ be a *-automorphism of $\mathcal{M}$ such that $\tau(\beta(x))=\tau(x)$ for every $x\in\mathcal{M}^+$ (i.e. $\beta$ is trace-preserving).

        Let $l^2(\mathbb{Z})$ denote the Hilbert space which consists of the complex-valued functions $f$ on $\mathbb{Z}$ which satisfy $\sum_{m\in\mathbb{Z}}|f(m)|^2<\infty$. Let $\{e_n\}_{n\in\mathcal{Z}}$ be the orthonormal basis of $l^2(\mathbb{Z})$ such that $e_n(m)=\delta(n,m)$.  We also denote the left regular representation of $\mathcal{Z}$ on $l^2(\mathbb{Z})$ by $\lambda:\mathbb{Z}\rightarrow B(l^2(\mathbb{Z}))$, where $\lambda(n)(e_m)=e_{m+n}$.

        We let $\mathcal{H}=L^2(\mathcal{M},\tau)\otimes l^2(\mathbb{Z})$, or equivalently, $H=\oplus_{m\in\mathbb{Z}} L^2(\mathcal{M},\tau)\otimes e_m$.  The representations $\Psi$ of $\mathcal{M}$ and $\Lambda$ of $\mathbb{Z}$ may be defined by
            \begin{align}
                &\Psi(x)(\xi\otimes e_m)=(\beta^{-m}\xi)\otimes e_m,  &\text{ for all } x\in\mathcal{M}, \xi \in L^2(\mathcal{M},\tau)  \text{ and } m\in\mathbb{Z} \notag \\
                &\Lambda (n)(\xi\otimes e_m)=\xi\times(\lambda(n)e_m)  &\text{ for all } n,m\in\mathbb{Z}. \notag
            \end{align}

        It is not hard to verify that
            $$\Lambda(n)\Psi(x)\Lambda(-n)=\Psi(\beta^n(x)) \qquad \text{ for all } x\in\mathcal{M} \text{ and }n\in\mathbb{Z}.$$

        We may define the crossed product of $\mathcal{M}$ by an action $\beta$, which we denote by $\mathcal{M}\rtimes_\beta \mathbb{Z}$, to be the von Neumann algebra generated by $\Psi(\mathcal{M})$ and $\Lambda(\mathbb{Z})$ in $B(\mathcal{H})$. When there is no possibility of confusion, we will identify $\mathcal{M}$ with its image $\Psi(\mathcal{M})$ under $\Psi$ in $\mathcal{M}\rtimes_\beta\mathbb{Z}$.

        In Chapter 13 of \cite{KR}, amongst others, it is shown that there exists a faithful, normal conditional expectation, $\Phi$, taking $\mathcal{M}\rtimes_\beta \mathbb{Z}$ onto $\mathcal{M}$ such that
            $$\Phi\bigg(\sum_{n=-N}^{N} \Lambda(n)\Psi(x_n)\bigg)=x_0 \qquad \text{ where } x_n\in\mathcal{M} \text{ for every } -N\leq n\leq N.$$
        There also exists a semifinite, normal, extended tracial weight on $\mathcal{M}\rtimes_\beta \mathbb{Z}$, which we still denote by $\tau$, and which satisfies
            $$\tau(y)=\tau(\Phi(y)), \text{ for every postive } y \in \mathcal{M}\rtimes_\beta \mathbb{Z}.$$

        \begin{example}\label{example6.3}
            Let $\mathcal{M}=l^\infty(\mathbb{Z})$.  Then $\mathcal{M}$ is an abelian von Neumann algebra with a semifinite, faithful, normal tracial weight, $\tau$ which is given by
                $$\tau(f)=\sum_{m\in\mathcal{Z}}f(m), \qquad \text{ for every positive } f\in l^\infty(\mathbb{Z}).$$
            We let $\beta$ be an action on $l^\infty(\mathcal{Z})$, which we define by
                $$\beta (f)(m)=f(m-1), \qquad \text{ for every } f\in l^\infty(\mathbb{Z}) \text{ and } m\in\mathbb{Z}.$$
            It is known (see, for example Proposition 8.6.4 of \cite{KR}) that $l^\infty(\mathbb{Z})\rtimes_\beta \mathbb{Z}$ is a type $I_\infty$ factor.  Therefore, for some separable Hilbert space $\mathcal{H}$, $l^\infty(\mathbb{Z})\rtimes_\beta \mathbb{Z}\simeq B(\mathcal{H})$.
        \end{example}

        The next result follows from our construction of crossed products. (See also section 3 of \cite{Arv}.)

        \begin{lemma} \label{lemma6.4}
            Consider the weak *-closed, non-self-adjoint subalgebra $\mathcal{M} \rtimes_\beta \mathbb{Z}_+$ of $\mathcal{M}\rtimes_\beta\mathbb{Z}$ which is generated by
                $$\{\Lambda(n)\Psi(x):x\in\mathcal{M}, n \geq 0\}.$$
            Then the following hold:
            \begin{enumerate}
                \item[(i)] $\mathcal{M}\rtimes_\beta \mathbb{Z}_+$ is a semifinite subdiagonal subalgebra with respect to $(\mathcal{M}\rtimes_\beta\mathbb{Z},\Phi)$.  We will denote such a semifinite subdiagonal subalgebra by $H^\infty$ and call $H^\infty$ an analytic crossed product.
                \item[(ii)] We denote by $H^\infty_0$ the space $\ker(\Phi)\cap H^\infty$.  Then $H^\infty_0$ is a weak *-closed nonself-adjoint subalgebra which is generated in $\mathcal{M}\rtimes_\beta \mathbb{Z}$ by $$\{\Lambda(n)\Phi(x) : x\in\mathcal{M}, n\geq 0\}$$ and satisfies $$H^\infty_0=\Lambda(1)H^\infty.$$
                \item[(iii)] $H^\infty \cap (H^\infty)^*=\mathcal{M}$.
            \end{enumerate}
        \end{lemma}

        We are able to characterize the invariant subspaces of a crossed product of a semifinite von Neumann algebra $\mathcal{M}$ by a trace-preserving action $\beta$.

        \begin{corollary}\label{corollary6.5}
            Suppose that $\mathcal{M}$ is a von Neumann algebra with a semifinite, faithful, normal tracial weight $\tau$.  Let $\alpha$ be a unitarily invariant, locally $\|\cdot\|_1$-dominating, mutually continuous norm with respect to $\tau$, and $\beta$ be a trace-preserving, *-automorphism of $\mathcal{M}$.  Consider the crossed product of $\mathcal{M}$ by an action $\beta$, $\mathcal{M}\rtimes_\beta \mathbb{Z}$. Still denote the semifinite, faithful, normal, extended tracial weight on $\mathcal{M}\rtimes_\beta \mathbb{Z}$ by $\tau$.

            Denote by $H^\infty$ the weak *-closed nonself-adjoint subalgebra in $\mathcal{M}\rtimes_\beta \mathbb{Z}$ which is generated by $\{\Lambda(n)\Psi(x) : x\in\mathcal{M}, n\geq 0\}$.  Then $H^\infty$ is a semifinite subdiagonal sublagebra of $\mathcal\rtimes_\beta \mathbb{Z}$.

            Let $\mathcal{K}$ be a closed subspace of $L^\alpha(\mathcal{M}\rtimes_\beta\mathbb{Z},\tau)$ such that $H^\infty \mathcal{K}\subseteq \mathcal{K}$.  Then there exist a projection $q$ in $\mathcal{M}$ and a family $\{u_\lambda\}_{\lambda\in\Lambda}$ of partial isometries in $\mathcal{M}\rtimes_\beta \mathbb{Z}$ which satisfy

            \begin{enumerate}
                \item[(i)] $u_\lambda q=0$ for all $\lambda\in\Lambda$;
                \item[(ii)] $u_\lambda u_{\lambda}^*\in\mathcal{M}$ and $u_\lambda u_{\mu}^*=0$ for all $\lambda,\mu\in\Lambda$ with $\lambda\neq \mu$;
                \item[(iii)] $\mathcal{K}=(L_\alpha(\mathcal{M}\rtimes_\beta \mathbb{Z})q)\otimes^{row}(\otimes^{row}_{\lambda\in\Lambda} H^\alpha u_\lambda)$.
            \end{enumerate}
        \end{corollary}

        \begin{proof}
            From Theorem \ref{theorem3.1}, we know that
                $$K=Y\oplus^{row}(\oplus^{row}_{\lambda\in\Lambda} H^\alpha u_\lambda)$$
            such that $Y$ is a closed subspace of $\mathcal{M}\rtimes_\beta \mathbb{Z}$ and a family of partial isometries, $\{u_\lambda\}$, in $\mathcal{M}\rtimes_\beta \mathbb{Z}$ which satisfy
            \begin{enumerate}
                \item[(a)] $u_\lambda Y^*=0$ for all $\lambda \in\Lambda$;
                \item[(b)] $u_\lambda u_{\lambda}^*\in\mathcal{M}$ and $u_\lambda u_{\mu}^* =0$ for all $\lambda, \mu \in \Lambda$ with $\lambda \neq \mu$;
                \item[(c)] $Y=[H^\infty_0 Y]_\alpha$.
            \end{enumerate}
            By Lemma \ref{lemma6.4} and (c), it is clear that
                $$Y =[H^\infty_0 Y]_\alpha=[\Lambda(1)H^\infty Y]_\alpha \subseteq \Lambda(1)Y.$$
            We can show, by induction, that $\Lambda(-n)Y\subseteq Y$ for any $n$ in $\mathbb{N}$.  From the defintion of $H^\infty$, we know that $\Lambda(n)Y\subset Y$ for every $n\geq 0$, and $\psi(x)Y\subseteq Y$ for every $x\in\mathcal{M}$.  Therefore, $Y\subseteq L^\alpha(\mathcal{M}\rtimes_\beta \mathbb{Z})$ is left $\mathcal{M}\rtimes_\beta \mathbb{Z}$-invariant, and from Corollary \ref{corollary5.5}, there exists a projection $q\in\mathcal{M}$ with $Y=L^\alpha(\mathcal{M}\rtimes_\beta \mathbb{Z}, \tau)q$.  Therefore,
            \begin{enumerate}
                \item[(i)] $u_\lambda q=0$ for all $\lambda\in\Lambda$;
                \item[(ii)] $u_\lambda u_{\lambda}^*\in\mathcal{M}$ and $u_\lambda u_{\mu}^*=0$ for all $\lambda,\mu\in\Lambda$ with $\lambda\neq \mu$;
                \item[(iii)] $\mathcal{K}=(L_\alpha(\mathcal{M}\rtimes_\beta \mathbb{Z})q)\otimes^{row}(\otimes^{row}_{\lambda\in\Lambda} H^\alpha u_\lambda)$
            \end{enumerate}
            hold, and the corollary is proven.
        \end{proof}

    \subsection{Invariant subspaces for $B(\mathcal{H})$}
        Let $\mathcal{H}$ be an infinite dimensional separable Hilbert space with orthonormal base $\{e_m\}_{m\in\mathbb{Z}}$.  We let $\tau=Tr$ be the usual trace on $B(\mathcal{H})$, namely
            $$\tau(x)=\sum_{m\in\mathcal{Z}}\langle xe_m,e_m\rangle\qquad \text{for every } x\in B(\mathcal{H}) \text{ with } x>0.$$
        With this $\tau$, $B(\mathcal{H})$ is a von Neumann algebra with a semifinite, faithful, normal tracial weight $\tau$.

        We let
            $$\mathcal{A}=\{x\in B(\mathcal{H}): \langle xe_m,e_n\rangle =0 \quad\forall n<m\}$$
        be the lower triangular subalgebra of $B(\mathcal{H})$.

        Recall from Example \ref{example6.3} that the crossed product of $l^\infty(\mathbb{Z})$ by an action $\beta$, denoted $l^\infty(\mathbb {Z})\rtimes _\beta \mathbb{Z}$, where the action $\beta$ is determined by
            $$\beta(f)(m)=f(m-1) \qquad \text{ for every } f\in l^\infty(\mathbb{Z}), m\in\mathbb{Z}$$
        is another way to realize $B(\mathcal{H})$.

        It is easy to see that $\mathcal{A}$ is $l^\infty(\mathbb{Z})\rtimes_\beta \mathbb{Z}_+$, a semifinite subdiagonal subalgebra of $l^\infty(\mathbb{Z}\rtimes_\beta\mathbb{Z}$ (see Lemma \ref{lemma6.4}).

        The following corollary follows from \ref{corollary6.5}.

        \begin{corollary}\label{corollary6.5.5} 
          Suppose $\mathcal{H}$ is a separable Hilbert space with an orthonormal base $\{e_m\}_{m\in\mathbb{Z}}$, and let
                $$H^\infty=\{x\in B(\mathcal{H}):\langle xe_m, e_n\rangle=0, \quad \forall n<m\}$$
            be the lower triangular subalgebra of $B(\mathcal{H})$. Then ${}\mathcal{D}=H^\infty \cap (H^\infty)^*$ is the diagonal subalgebra of $B(\mathcal{H})$. Suppose $\alpha:\mathcal{I}\rightarrow[0,\infty)$, where $\mathcal{I}$ is the set of elementary operators in $\mathcal{M}$, is an unitarily invariant norm such that any net $\{e_\lambda\}$ in $\mathcal{M}$ with $e_\lambda \uparrow I$ in the weak* topology implies that $\alpha((e_\lambda - I)x)\rightarrow 0$.

            Assume that $\mathcal{K}$ is a closed subspace of $H^\alpha$ such that $H^\infty\mathcal{K}\subseteq \mathcal{K}$. Then there exists a projection $q$ in $\mathcal{D}$ and $\{u_\lambda\}_{\lambda\in\Lambda}$, a family of partial isometries in $H^\infty$ which satisfy
              \begin{enumerate}
                \item[(i)] $u_\lambda q=0$ for every $\lambda\in\Lambda$;
                \item[(ii)] $u_\lambda u_\lambda^* \in \mathcal{D}$, and $u_\lambda u_\mu^*=0$ for every $\lambda,\mu\in\mathcal{D}$ with $\lambda\neq\mu$;
                \item[(iii)] $\mathcal{K}=(B(\mathcal{H})q)\oplus^{row}(\oplus^{row}_{\lambda\in\Lambda}H^\alpha u_\lambda)$. 
              \end{enumerate}
          \end{corollary}

        The following is a corollary of Theorem \ref{corollary6.5} and proposition \ref{prop3.5}.

        \begin{corollary}\label{corollary6.6}
            Suppose $\mathcal{H}$ is a separable Hilbert space with an orthonormal base $\{e_m\}_{m\in\mathbb{Z}}$, and let
                $$H^\infty=\{x\in B(\mathcal{H}):\langle xe_m, e_n\rangle=0, \quad \forall n<m\}$$
            be the lower triangular subalgebra of $B(\mathcal{H})$. Then $\mathcal{D}=H^\infty \cap (H^\infty)^*$ is the diagonal subalgebra of $B(\mathcal{H})$. Suppose $\alpha:\mathcal{I}\rightarrow[0,\infty)$, where $\mathcal{I}$ is the set of elementary operators in $\mathcal{M}$, is an unitarily invariant norm such that any net $\{e_\lambda\}$ in $\mathcal{M}$ with $e_\lambda \uparrow I$ in the weak* topology implies that $\alpha((e_\lambda - I)x)\rightarrow 0$.

            Assume that $\mathcal{K}$ is a closed subspace of $H^\alpha$ such that $H^\infty\mathcal{K}\subseteq \mathcal{K}$.  Then there exists $\{u_\lambda\}_{\lambda\in\Lambda}$, a family of partial isometries in $H^\infty$ which satisfy
            \begin{enumerate}
                \item[(i)] $u_\lambda u_\lambda^*\in\mathcal{D}$ and $u_\lambda u_\mu^*=0$ for every $\lambda, \mu \in \Lambda$ such that $\lambda\neq\mu$;
                \item[(ii)] $\mathcal{K}=\oplus^{row}_{\lambda\in\Lambda} H^\alpha u_\lambda$.
            \end{enumerate}
        \end{corollary}

        \begin{remark}
            The result is similar when $H^\infty$ is instead the upper triangular subalgebra of $B(\mathcal{H})$.
        \end{remark}

        \begin{remark}
            Recall that any unitarily invariant norm $\alpha$ gives rise to a symmetric gauge norm $\Psi$ on the spectrum of $|A|$, $\{a_n\}_{1\leq n\leq N}$, where $A$ is a finite rank operator.  Then Corollary \ref{corollary6.6} holds for $\Psi$.
        \end{remark}


\section*{Acknowledgements}
The research in this paper was from a thesis submitted to the Graduate School at the University of New Hampshire as part of the requirements for completion of a doctoral degree.


\begin{thebibliography}{99}                                                                                               %
\bibitem {Arv}W. B. Arveson, \emph{Analyticity in operator algebras}, Amer. J.
Math. 89 (1967) 578--642.

\bibitem {Be} T. Bekjan, \emph{ Noncommutative Hardy space associated with semi-finite subdiagonal algebras}, J. Math. Anal. Appl. 429 (2015), no. 2,
1347–1369.

\bibitem {BX}T. N. Bekjan and Q. Xu, \emph{Riesz and Szeg\"{o} type
factorizations for noncommutative Hardy spaces}, J. Operator Theory,
62 (2009) 215-231.

 \bibitem {Be1}T. N. Bekjan, \emph{Noncommutative symmetric Hardy spaces}, Integr.
 Equ. Oper. Theory 81 (2015) 191-212.


\bibitem {B}A. Beurling, \emph{On two problems concerning linear
transformations in Hilbert space}, Acta Math. 81 (1949) 239-255.


\bibitem {BL2}D. Blecher and L. E. Labuschagne, \emph{A Beurling theorem for
noncommutative $L^{p}$}, J. Operator Theory, 59 (2008) 29-51.



 \bibitem {Bo}S. Bochner, \emph{Generalized conjugate and analytic functions
 without expansions}, Proc. Nat. Acad. Sci. U.S.A. 45 (1959) 855-857.

 \bibitem {CHS} Y. Chen, D. Hadwin and J. Shen, \emph{A non-commutative Beurling's theorem with respect to unitarily invariant norms}, J. Operator Theory, 75 (2016) 497-523.

\bibitem {P2}P. Dodds and T. Dodds, \emph{Some properties of symmetric
operator spaces}, Proc. Centre Math. Appl. Austral. Nat. Univ., 29,
Austral. Nat. Univ., Canberra, 1992.

\bibitem {P}P. Dodds, T. Dodds and B. Pagter, \emph{Noncommutative Banach
function spaces}, Mathematische Zeitschrift, 201 (1989) 583-597.



\bibitem {Do}P. Dodds, T. Dodds and B. Pagter, \emph{Noncommutative K\"{o}the
duality}, Trans. Amer. Math. Soc. 339 (1993) 717-750.

\bibitem {Exel}R. Exel, \emph{Maximal subdiagonal algebras}, Amer. J. Math.
110 (1988) 775-782.

\bibitem {Fa}T. Fack and H. Kosaki, \emph{Generalized $s$-numbers of $\tau$-measurable
operators}, Pacific J. Math. 2 (1986) 269-300.



\bibitem {Fang}J. Fang, D. Hadwin, E. Nordgren and J. Shen, \emph{Tracial
gauge norms on finite von Neumann algebras satisfying the weak
Dixmier property}, J. Funct. Anal. 255 (2008) 142-183.

\bibitem{GK} I. C. Goldberg and M. G. Krein, {\em Introduction to the theory of linear nonselfadjoint operators}, Vol. 18, Translations of Mathematical Monographs, 1969.



\bibitem {Halmos}P. Halmos, \emph{Shifts on Hilbert spaces}, J. Reine Angew.
Math. 208 (1961) 102-112.

\bibitem {He}H. Helson, \emph{Lectures on Invariant Subspaces}, Academic
Press, New York-London, 1964.

\bibitem {HL}H. Helson and D. Lowdenslager, \emph{Prediction theory and
Fourier series in several variables}, Acta Math. 99 (1958) 165-202.

\bibitem {HR} E. Hewitt and K. A. Ross, {\em Abstract harmonic analysis}, Vol. 2, Springer-Verlag, Berlin, 1970

\bibitem {Ho}K. Hoffman, \emph{Analytic functions and logmodular Banach
algebras}, Acta Math. 108 (1962) 271-317.


\bibitem {KR} R. Kadison and J. Ringrose, \emph{Fundamentals of the thoery of operator algebras, volume II, advanced theory}, Academic Press, Inc, (1986).

\bibitem {Kunze}R. A. Kunze, \emph{$L^{p}$-Fourier transforms on locally
compact unimodular groups}, Trans. Amer. Math. Soc. 89 (1958)
519-540.

\bibitem{Ji} G. Ji, \emph{Maximality of semi-finite subdiagonal algebras}, J. Shaanxi Normal Univ. Sci. Ed. 28 (2000) 15-17.

\bibitem {JS}M. Junge and D. Sherman, \emph{Noncommutative $L^{p}$-modules}, J. Operator Theory 53 (2005) 3-34.

\bibitem {MG}M. Marsalli and G. West, \emph{Noncommutative $H^{p}$ spaces }, J. Operator Theory  40 (1998) 339-355.



\bibitem {McCarthy}C. A. McCarthy, \emph{$C_{p}$}, Israel J. Math. 5 (1967) 249-271.

\bibitem {MMS} M. McAsey, P. Muhly and K. Saito, \emph{Nonselfadjoint Crossed Products (Invariant Subspaces and Maximality)},  Transactions of the American Mathematical Society, Vol. 248, No. 2 (Mar., 1979), pp. 381
-409.

\bibitem{MvN} F. Murray and J. von Neumann, \emph{On rings of operators, IV}, Annals of Mathematics, Vol. 44, No. 2 second series, pp. 716-808. 


\bibitem  {NW} T. Nakazi and Y. Watatani, {\em Invariant subspace theorems for subdiagonal algebras}, J. Operator Theory 37(1997), 379-395.


\bibitem {Nelson}E. Nelson, \emph{Notes on noncommutative integration}, J.
Funct. Anal. 15 (1974) 103-116.




\bibitem {vNeumann}J. von Neumann, \emph{Some matrix-inequalities and
metrization of matric-space}, Tomsk Univ. Rev. 1 (1937) 286-300.

\bibitem {depag} B. de Pagter, \emph {Non-commutative Banach function spaces}, K. Boulabiar, G. Buskes, A. Triki (Eds.), Positivity, Trends in Mathematics, Birkhäuser Verlag, Basel (2007), pp. 197–227


\bibitem {PX}G. Pisier and Q. Xu, \emph{Noncommutative $L^{p}$-spaces},
Handbook of the geometry of Banach spaces, North-Holland, Amsterdam,
2 (2003) 1459-1517.
\bibitem {Sakai} Sakai, \emph{C$^*$-algebras and W$^*$-algebras},
Springer, 1971.

\bibitem {Sag} L. Sager, \emph{A Beurling-Blecher-Labuschagne theorem for noncommutative Hardy spaces associated with semifinite von Neumann algebras}, Integr. Equ. Op. Theory 86 (2016) 377-407.  doi: 10.1007/s00020-016-2308-z

\bibitem {Sai}K. S. Saito, \emph{A note on invariant subspaces for finite
maximal subdiagonal algebras}, Proc. Amer. Math. Soc. 77 (1979)
348-352.


\bibitem {Sai2}  K. S. Saito, \emph{A simple approach to the invariant subspace structure of analytic crossed products,}
 J. Operator Theory 27 (1992), no. 1, 169-177.


\bibitem {Schatten} R. Schatten, {\em A Theory of Cross-Spaces}, Annals of Mathematics Studies, no. 26, Princeton University Press, Princeton, New Jersey, 1950.

\bibitem {Schatten1} R. Schatten, {\em Norm ideals of completely continuous operators}, Springer-Verlag, Berlin, 1960.


\bibitem {Se}I. Segal, \emph{A noncommutative extension of abstract
integration}, Ann. Math. 57 (1952) 401-457.

\bibitem {Simon}B. Simon, \emph{Trace ideals and their applications}, London
Mathematical Society Lecture Note Series, vol. 35, Cambridge
University Press, Cambridge-New York, 1979.

\bibitem {Sr}T. P. Srinivasan, \emph{Simply invariant subspaces}, Bull. Amer.
Math. Soc. 69 (1963) 706-709.

\bibitem {SW}T. Srinivasan and J. K. Wang, \emph{Weak*-Dirichlet algebras},
Proceedings of the International Symposium on Function Algebras,
Tulane University, 1965 (Chicago), Scott-Foresman, 1966, 216--249.

\bibitem {Ta}M. Takesaki, \emph{Theory of Operator Algebras I}, Springer, 1979.

\bibitem{Xu} Q. Xu, \emph{On the maximality of subdiagonal algebras}, J. Operator Theory 54 (2005), no. 1, 137-146.

\bibitem{Xu2}  Q. Xu \emph{Operator spaces and noncommutative
L$^p$, The part on noncommutative L$^p$-spaces} Lectures in the
Summer School on Banach spaces and Operator spaces, Nankai
University China July 16 - July 20, 2007.

\bibitem {Y}F. Yeadon, \emph{Noncommutative $L^{p}$-spaces}, Math. Proc.
Cambridge Philos. Soc. 77 (1975) 91-102.
\end{thebibliography}
\end{document}